\documentclass{article}

\usepackage{fullpage}
\usepackage{authblk}

\usepackage[utf8]{inputenc}
\usepackage{amsmath, amsfonts, amssymb, amsthm}
\usepackage{enumerate}
\usepackage{graphicx}
\usepackage{url}
\usepackage{csquotes}
\usepackage[numbers]{natbib}

\newtheorem{theorem}{Theorem}[section]
\newtheorem{lemma}[theorem]{Lemma}

\newtheorem{corollary}[theorem]{Corollary}
\newtheorem{proposition}[theorem]{Proposition}

\newcommand\blfootnote[1]{%
  \begingroup
  \renewcommand\thefootnote{}\footnote{#1}%
  \addtocounter{footnote}{-1}%
  \endgroup
}

\title{Exploring spaces of semi-directed phylogenetic networks}

\author[1]{Simone Linz}
\author[2$^\ast$]{Kristina Wicke}

\affil[1]{School of Computer Science, University of Auckland, Auckland, New Zealand}
\affil[2]{Department of Mathematical Sciences, New Jersey Institute of Technology, Newark NJ, USA}

\date{}

\begin{document}
\maketitle

\begin{abstract}
    Semi-directed phylogenetic networks have recently emerged as a class of phylogenetic networks sitting between rooted (directed) and unrooted (undirected) phylogenetic networks as they contain both directed as well as undirected edges. While the spaces of rooted phylogenetic networks and unrooted phylogenetic networks have been analyzed in recent years and various rearrangement moves to traverse these spaces have been introduced, the results do not immediately carry over to semi-directed phylogenetic networks. Here, we propose a simple rearrangement move for semi-directed phylogenetic networks, called cut edge transfer (CET), and show that the space of semi-directed level-$1$ networks with precisely $k$ reticulations is connected under CET. This level-$1$ space is currently the predominantly used search space for most algorithms that reconstruct semi-directed phylogenetic networks. Hence, every semi-directed level-$1$ network with a fixed number of reticulations and leaf set can be reached from any other such network by a sequence of CETs. By introducing two additional moves, CET$^+$ and CET$^-$, that allow for the addition or deletion of reticulations, we then establish connectedness for the space of all semi-directed phylogenetic networks on a fixed leaf set. As a byproduct of our results for  semi-directed phylogenetic networks, we also show that the space of rooted level-$1$ networks with a fixed number of reticulations and leaf set is connected under CET, when translated into the rooted setting.
\end{abstract}


\blfootnote{$^\ast$Corresponding author\\ \textit{Email addresses:} \url{s.linz@auckland.ac.nz}, \url{kristina.wicke@njit.edu}}

\section{Introduction} 
Phylogenetic networks are a generalization of phylogenetic trees allowing for the representation of speciation and reticulate evolutionary events such as hybridization or lateral gene transfer. Traditionally, two types of phylogenetic networks were considered in the literature: unrooted (also referred to as undirected or implicit) phylogenetic networks and rooted (also referred to as directed or explicit) phylogenetic networks (see for example~\cite{huson2010phylogenetic}). While the former are often used to represent conflict in data and lack evolutionary directionality, the latter explicitly depict evolution as a directed process from some common ancestor that is represented by the root to the present-day species that are represented by the leaves of the network. Importantly, rooted phylogenetic networks are rooted directed acyclic graphs that, in comparison with phylogenetic trees, contain vertices with in-degree at least two that represent reticulation events. 

Recently, a class of phylogenetic networks that have directed and undirected edges, called {\it semi-directed phylogenetic networks}, has emerged in the literature. Roughly speaking, semi-directed phylogenetic networks are obtained from rooted phylogenetic networks by suppressing the root whose position is not identifiable under many models of sequence evolution and ignoring the direction of all edges, except for those directed into a vertex of in-degree at least two, thereby keeping information on which vertices represent reticulation events. A formal definition of semi-directed phylogenetic networks and all other mathematical concepts used in this paper is given in the next section.

Semi-directed phylogenetic networks have been the focus on studies concerning identifiability (see, e.g., \cite{Allman2022, Ardiyansyah2021, Banos2018, Gross2018, Gross2020, Hollering2021, SolisLemus2016, Solis-Lemus2020, Xu2021}) and also play a major role in phylogenetic network estimation algorithms such as NANUQ~\cite{allman2019nanuq}, SNaQ \cite{SolisLemus2016}, and PhyNEST~\cite{kong2022inference}. The latter two find an optimal semi-directed phylogenetic network that best ``fits'' the observed data under a composite likelihood (also called pseudo-likelihood) framework and search through a space of semi-directed phylogenetic networks (detailed below). While SNaQ is implemented in the popular software tool PhyloNetworks \cite{SolisLemus2017} and uses gene trees and quartet concordance factors  as input, PhyNEST reconstructs an optimal network from site patterns. Like the reconstruction of rooted and unrooted phylogenetic networks, the reconstruction of an optimal semi-directed phylogenetic network typically involves searching the space of all semi-directed phylogenetic networks on a fixed leaf set. More specifically, given an initial phylogenetic network, the network is modified by locally rearranging its structure, the fit of the new network is evaluated, and if there is an improvement in fit, the search continues from that network until a local optimum is found. This  strategy is  referred to as hill-climbing.  Although alternative optimization strategies such as simulated annealing exist, they all involve the need of traversing spaces of phylogenetic networks.

A fundamental question that arises in this regard is whether the space of phylogenetic networks is connected under a given rearrangement operation. In other words, can every phylogenetic network of a space of networks (e.g., all semi-directed phylogenetic networks on a fixed leaf set) be reached from any other phylogenetic network in the space by applying a sequence of these rearrangement operations such that the resulting network after each operation is also in the space? This question has been analyzed for various spaces of unrooted and rooted phylogenetic trees (e.g.,~\cite{allen2001subtree,bordewich2005computational,hein1996complexity}), unrooted phylogenetic networks (e.g.,~\cite{Huber2015, Huber2016, Francis2017, Janssen2019}) and rooted phylogenetic networks (e.g., \cite{bordewich2017lost, Erdos2021, Gambette2017, Janssen2021, Janssen2018, Klawitter2018}), and several rearrangement moves to traverse these spaces have been introduced. However, these results do not immediately carry over to spaces of semi-directed phylogenetic networks. With a focus on the reconstruction of semi-directed level-$1$ networks which are networks whose underlying cycles are vertex disjoint, the authors of \cite{SolisLemus2016} suggested that the moves employed in SNaQ assure connectivity due to their similarity to moves for which there is an established connectivity result for unrooted level-$1$ networks \cite{Huber2015}. However, this has not been formally proven yet. Indeed, Figure 1 of~\cite{Huber2015} shows that, although the space of all unrooted level-$1$ networks on four leaves is connected under the operation proposed in that paper, the space of all such network restricted to those with a single reticulation is not connected under the same operation.

The main purpose of this paper is to establish rigorous connectivity results for spaces of semi-directed phylogenetic networks. Since SNaQ~\cite{SolisLemus2016} and other algorithms in this area of research such as NANUQ~\cite{allman2019nanuq} and PhyNEST~\cite{kong2022inference} focus on the reconstruction of semi-directed level-$1$ networks, we also focus on establishing connectivity results for spaces of semi-directed level-$1$ networks. To this end, we propose a new rearrangement operation for semi-directed phylogenetic networks, called cut edge transfer (CET), which prunes a subnetwork of a semi-directed phylogenetic network by deleting a cut edge and reconnects the two smaller networks by adjoining them with a new cut edge.  We then prove that, under CET, the space of semi-directed level-$1$ networks with a fixed number $k$ of reticulations and leaf set $X$ is connected. Hence, every semi-directed level-$1$ network with $k$ reticulations and leaf set $X$ can be reached from any other such network by a sequence of CETs such that the network resulting from each CET in the sequence is also a semi-directed level-$1$ network with $k$ reticulations and leaf set $X$. As a byproduct of our results, we establish connectivity of rooted level-$1$ networks with a fixed number of reticulations and leaf set under a rooted version of CET. While CETs operate on semi-directed networks of the same \enquote{reticulate complexity} (i.e., the same number of reticulations), we additionally introduce two moves CET$^+$ and CET$^-$ that allow for a change in the number of reticulations by one. Here, we show that under CET, CET$^+$, and CET$^-$, the space of all semi-directed phylogenetic networks on a fixed leaf set and the space of all semi-directed level-$1$ networks with a fixed leaf set are connected. Lastly, we show that if two semi-directed level-$1$ networks are connected by a single CET, then they are also connected by a sequence of restricted local CETs. Such a restricted CET, to which we refer to as CET$_1$, moves a pruned subnetwork across a single internal edge. This last result suggests that the rearrangement moves employed in SNaQ~\cite{SolisLemus2016} are sufficient to reach any semi-directed level-$1$ network in the search space if their so-called ``nearest neighbor interchange (NNI) move on a tree edge''  is slightly relaxed to allow for NNI moves on undirected and directed edges.

The remainder of the paper is organized as follows. We begin by defining rooted and semi-directed phylogenetic networks, as well as several  concepts in the study of phylogenetic networks in Section \ref{sec:prelim}. In Section \ref{sec:CET} we introduce the CET operation and discuss some of its properties. Subsequently, in Section~\ref{sec:rooted} we establish connectedness results for spaces of rooted level-$1$ networks under CET that play a crucial role in establishing analogous results for spaces of semi-directed level-$1$ networks. In Section~\ref{sec:connectedness}, we finally turn to semi-directed phylogenetic networks. We first establish connectedness of semi-directed level-$1$ networks with a fixed number of reticulations and leaf set in Section~\ref{sec:withintier} and then connectedness for all such networks if only the leaf set is fixed in Section \ref{sec:acrosstiers}. Lastly, in Section \ref{Sec:CET1} we show that if two semi-directed level-$1$ networks are connected by a single CET, then they are also connected by a sequence of local CET$_1$ moves. We end the paper with some concluding remarks and directions for future research in Section \ref{sec:conclusion}.

\section{Preliminaries} \label{sec:prelim}

Throughout this paper, $X$ denotes a non-empty finite set.\\

\noindent {\bf Phylogenetic networks.} A {\em rooted binary phylogenetic network $N_r$ on $X$} is a rooted acyclic directed graph with no loops that satisfies the following three properties:
\begin{enumerate}[(i)]
\item the (unique) root $\rho$ has in-degree zero and out-degree one; 
\item a vertex of out-degree zero has in-degree one, and the set of vertices with out-degree zero is $X$; and
\item all other vertices have either in-degree one and out-degree two, or in-degree two and out-degree one.
\end{enumerate}
The set $X$ is called the {\it leaf set} of $N_r$. As with other publications on spaces of phylogenetic networks~\cite{bordewich2017lost,Janssen2019}, we allow edges to be in parallel or, equivalently, underlying cycles of length two. A vertex with in-degree two and out-degree one is called a {\it reticulation}, and a vertex with in-degree one and out-degree two is called a {\it tree vertex}. Similarly, an edge directed into a reticulation is called a {\it reticulation edge} and each non-reticulation edge is called a {\it tree edge}. Lastly, for two vertices $u$ and $v$, we say that $u$ is a {\em parent} of $v$ and $v$ is a {\em child} of $u$ if $(u,v)$ is an edge of $N_r$. 

A rooted binary phylogenetic $X$-tree $T$ is a rooted binary phylogenetic network on $X$ with no reticulation. Furthermore, a rooted binary phylogenetic $X$-tree $T$ with $\vert X \vert=n$ is called a {\it caterpillar} if $n=1$, or if $n\ge 2$ and we can order the elements in $X$, say $x_1, x_2, \ldots, x_n$, so that $x_1$ and $x_2$ have the same parent and, for all $i\in \{2, 3, \ldots, n-1\}$, we have that $(p_{i+1}, p_i)$ is an edge in $T$, where $p_{i+1}$ and $p_i$ are the parents of $x_{i+1}$ and $x_i$, respectively. We denote such a caterpillar $T$ by $(x_1,x_2,x_3\ldots,x_n)$ or, equivalently, $(x_2, x_1, x_3,\ldots,x_n)$.

We next define a second network type that will play an important role in this paper and  has directed and undirected edges. Following the definition that is used in~\cite{SolisLemus2016}, we say that a network $N_s$ with leaf set $X$ is a {\it semi-directed binary phylogenetic network} on $X$ if it can be obtained from a rooted binary phylogenetic network $N_r$ on $X$ by deleting $\rho$, suppressing the resulting vertex of in-degree zero and out-degree two, and undirecting all tree edges, in which case $N_r$ is called a {\it rooted partner} of $N_s$. Importantly, $N_r$ is not necessarily the only rooted partner of $N_s$.

Now, let $N_s$ be a semi-directed binary phylogenetic network.  
Observe that $N_s$ may have a single (directed) loop in which case each rooted partner of $N_s$ has two edges in parallel that are incident with the (unique) child of its root. Keeping the exceptional loop in mind, we call a vertex $v$ of $N_s$ a {\it reticulation} if there either exist two edges that are directed into $v$ or $(v,v)$ is a loop. Furthermore, each edge of $N_s$ that is directed is called a {\it reticulation edge}.

Let $N_s$ and $N_s'$ be two semi-directed binary phylogenetic networks on $X$ with vertex and edge sets $V$ and $E$, and $V'$ and $E'$, respectively. Then $N_s$ and $N_s'$ are {\it isomorphic} if there is a bijection $\psi :V\rightarrow V'$ such that $\psi(x)=x$ for all $x\in X$ and $(u,v) \in E$ (resp.$\{u,v\} \in E$) if and only if $(\psi(u),\psi(v)) \in E'$ (resp. $\{\psi(u),\psi(v)\} \in E'$) for all $u,v\in V$. If $N_s$ and $N_s'$ are isomorphic, we write $N_s\cong N_s'$ and, otherwise, we write $N_s\ncong N_s'$.

For the remainder of the paper, we will refer to the two types of  rooted binary phylogenetic networks and semi-directed binary phylogenetic networks as {\it rooted phylogenetic networks} and {\it semi-directed phylogenetic networks}, respectively, as all such networks considered here are binary. Moreover, whenever we use the expression of a {\it phylogenetic network $N$} without specifying a type, then the following statement or definition applies to both types of networks. 
Additionally, in all figures except for Figure~\ref{Fig_NrNsNu}, the edges of rooted phylogenetic networks are directed down the page and we omit arrowheads. 

Let $N$ be a phylogenetic network. 
Recall that $N$ may have a loop if it is semi-directed. For $\ell\geq 1$, we refer to a sequence $v_1,v_2,\ldots,v_\ell$
of $\ell$ distinct vertices of $N$ as a {\em cycle of length $\ell$} or as an {\em $\ell$-cycle}  if $\{v_\ell,v_1\}$ and, for each $i\in\{1,2,\ldots,\ell-1\}$, $\{v_i,v_{i+1}\}$ are edges in the underlying graph of $N$. If $\ell=1$, the definition of a cycle of length one coincides with that of a loop. Furthermore, if the length of an $\ell$-cycle is irrelevant, we simply refer to it as a {\em cycle}. 
Finally, the number of reticulations of $N$ is denoted by $r(N)$.\\

\noindent {\bf Level-$1$ networks.} Let $N_r$ be a rooted phylogenetic network. Then $N_r$ is said to be {\it level-$1$} if each  cycle has length at least three and no two  cycles have a common vertex. Moreover, if $N_r$ is a rooted level-$1$ network and $v$ is a vertex of a cycle $C$ of $N$, we call $v$ the {\it source} of $C$ if no edge of $N_r$ that is directed into $v$ lies on $C$.  If, on the other hand, $v$ is the unique reticulation of $C$, then we call it the {\it sink} of $C$. Since $N_r$ is level-$1$, each  cycle of $N$ has a unique source and sink. 

Extending the definition of level-$1$ to a semi-directed phylogenetic network $N_s$, we say that $N_s$ is {\it level-$1$} if there exists a rooted partner of $N_s$ that is level-$1$. 
Notice that a semi-directed level-$1$ network may contain one pair of parallel edges. This is the case if it was obtained from a rooted level-$1$ network with the property that the unique child of the root is the source of a cycle of length three. An example of this is depicted in Figure \ref{Fig_NrNsNu}.

\begin{figure}[t]
    \centering
    \includegraphics[scale=0.3]{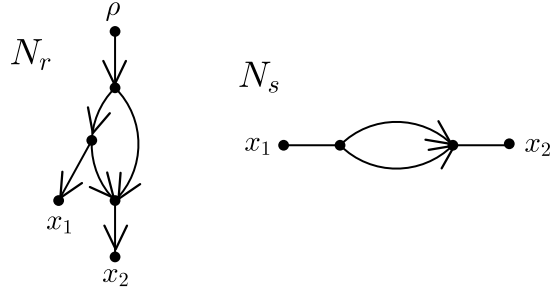}
    \caption{A semi-directed phylogenetic network $N_s$ on $X=\{x_1,x_2\}$ and a rooted partner $N_r$ of $N_s$.  
    As the child of the root of $N_r$ is the source of a cycle of length three,  $N_s$ contains a pair of parallel edges. 
    Each of $N_r$ and $N_s$ is level-$1$.}
    \label{Fig_NrNsNu}
\end{figure}

We end this section by bounding the number of reticulations in rooted and semi-directed level-$1$ networks.
\begin{lemma}\label{l:reticulations}
Let $N$ be a rooted or semi-directed level-$1$ network on $X$. Then $N$ has at most $\vert X \vert-1$ reticulations.
\end{lemma}

\begin{proof}
First, suppose that $N$ is a rooted level-$1$ network. Then the lemma follows from~\cite{cardona2008comparison,mcdiarmid2015counting} and the fact that each level-$1$ network is also tree-child~\cite{huber2022orienting}. Second, suppose that $N$ is a semi-directed level-$1$ network. Let $N_r$ be a rooted partner of $N$ that is level-$1$. By construction, $v$ is a reticulation in $N$ if and only if $v$ is a reticulation in $N_r$. As, $N_r$ has at most $\vert X \vert-1$ reticulations, so does $N$.
\end{proof}

\section{Cut edge transfers}\label{sec:CET}
In this section we introduce a new rearrangement operation that can be applied to phylogenetic networks and that will play a crucial role in establishing that the space of semi-directed level-$1$ networks on a fixed leaf set is connected.

Let $N$ be a phylogenetic network, and let $e$ be an edge of $N$. Recalling that $N$ is binary, $e$ is called a {\it cut edge} of $N$ if the deletion of $e$ from $N$ results in a graph with exactly  two connected components.

Now let $N_r$ be a rooted phylogenetic network, and let $e=(u,v)$ be a cut edge of $N_r$ such that $e$ is not incident with $\rho$ and $u$ is not a reticulation. Obtain a network $N_r'$ from $N_r$ by deleting $e$, suppressing $u$, subdividing an edge of the connected component that contains $\rho$ with a new vertex $u'$, and adding a new edge $(u',v)$. Clearly, $N_r'$ is a rooted phylogenetic network on $X$. If $N_r\ncong N_r'$, we say that $N_r'$ is obtained from $N_r$ by a single {\it cut edge transfer (CET)}. Furthermore, if $N_r'$ can be obtained from $N_r$ by a single CET, then conversely $N_r$ can also be obtained from $N_r'$ by the single CET that reverses the roles of $u$ and $u'$. Hence, any CET is reversible. Lastly, if $N_r$ is a rooted phylogenetic $X$-tree, then CETs coincide with rooted subtree prune and regraft (rSPR) operations~\cite{bordewich2005computational}.

We now turn to semi-directed phylogenetic networks and establish a relationship between cut edges and reticulation edges of such networks. 

\begin{lemma}\label{l:ret-edges}
Let $N_s$ be a semi-directed phylogenetic network, and let $e$ be an edge of $N_s$. If $e$ is a reticulation edge of $N_s$, then $e$ is an edge of a cycle in $N_s$. Moreover, no cut edge of $N_s$ is a reticulation edge. 
\end{lemma}

\begin{proof}
Let $N_r$ be a rooted partner of $N_s$. Suppose that $e$ is a reticulation edge of $N_s$. By construction of $N_s$ from $N_r$, it follows that, as $e$ is an edge of a cycle in $N_r$, $e$ is also an edge of a cycle in $N_s$. Now, let $f$ be a cut edge of $N_s$. Since $f$ is not an edge of a cycle, $f$ is not a reticulation edge of $N_s$.
\end{proof}

We next extend the definition of a CET to semi-directed phylogenetic networks. To this end, we first establish a lemma that pinpoints the relationship between cut edges of a semi-directed phylogenetic network and those of a rooted partner.

\begin{lemma}\label{l:cut-edge}
Let $N_s$ be a semi-directed phylogenetic network, and let $N_r$ be a rooted partner of $N_s$ with root $\rho$. Let $u$ and $v$ be two vertices of $N_s$. Then $e=\{u,v\}$ is a cut edge of $N_s$ if and only if exactly one of the following two conditions applies:
\begin{enumerate}[(i)]
\item $(u,v)$ or $(v,u)$ is a cut edge of $N_r$, or
\item $(\rho,t)$, $(t,u)$, and $(t,v)$ are cut edges of $N_r$, where $t$ is the unique child of $\rho$.
\end{enumerate}
\end{lemma}

\begin{proof}
Let $t$ be the unique child of $\rho$ in $N_r$. By construction of $N_s$ from $N_r$ it follows that $\{u,v\}\cap\{\rho,t\}=\emptyset$.
First, suppose that $e=\{u,v\}$ is a cut edge of $N_s$. If (i) does not apply, then, by construction of $N_s$ from $N_r$, it follows that neither $(u,v)$ nor $(v,u)$ is an edge of $N_r$. Hence, $t$ is the parent of each of $u$ and $v$ in $N_r$; thereby implying that (ii) holds. 

Second, suppose that one of (i) and (ii) applies. Clearly, if (i) applies, then $\{u,v\}$ is a cut edge of $N_s$. On the other hand, if (ii) applies, then it again follows from the construction of $N_s$ from $N_r$ that $\{u,v\}$ is a cut edge of $N_s$.
\end{proof}

Now, let $N_s$ be a semi-directed phylogenetic network on $X$.
Let $e=\{u,v\}$ be a cut edge of 
$N_s$ such that $u$ is not a reticulation and there exists a rooted partner $N_r$ of $N_s$ that satisfies one of the following two conditions.
\begin{enumerate}
\item $u$ is the parent of $v$ in $N_r$ or
\item there exist three cut edges $(\rho,t)$, $(t,u)$, and $(t,v)$ in $N_r$.
\end{enumerate}
Observe that, by Lemma~\ref{l:cut-edge}, these are the only two possibilities.
Then obtain a network $N_s'$ from $N_s$ by deleting $e$, suppressing $u$,
subdividing an edge of the connected component that does not contain $v$ with a new vertex $u'$, and adding a new edge $\{u',v\}$. If $u'$ subdivides a loop $(w,w)$ of $N_s$, then the resulting undirected edge $\{u',w\}$ is additionally directed towards $w$ so that $N_s'$ has two  parallel edges $(u',w)$. To see that $N_s'$ is a semi-directed phylogenetic network, observe the following. 
If the connected component containing $v$ does not contain any cycle, then the operation described above clearly preserves the fact that the edges of the resulting graph can be directed to yield a rooted phylogenetic network, which implies that $N_s'$ has a rooted partner. If, on the other hand, the connected component containing $v$ contains a cycle, then, by the choice of $u$ and $v$, there exists a rooted partner $N_r$ of $N_s$ satisfying Conditions 1. or 2. given above. In particular, all edges in the connected component of $N_s$ that contains $v$, must be directed away from $v$ in $N_r$. So again, the described operation results in a graph that can  be directed to yield a rooted phylogenetic network, implying that, in both cases, $N_s'$ is a semi-directed phylogenetic network. If $N_s\ncong N_s'$, we say that $N_s'$ is obtained from $N_s$ by a single {\it cut edge transfer (CET)}. Similar to the rooted case, if $N_s'$ can be obtained from $N_s$ by a single CET, then conversely $N_s$ can also be obtained from $N_s'$ by a single CET. 

We remark that carefully choosing a cut edge $e=\{u,v\}$ in the definition of a CET is crucial to ensure that the CET results in a semi-directed phylogenetic network. For arbitrary choices of $u$ and $v$, a CET may result in a graph that is not a semi-directed phylogenetic network. An example is depicted in Figure \ref{fig:CETsemidirected}. In this example, $N_s'$ is obtained from $N_s$ by a single CET that is applied to the cut edge $\{u,v\}$. This is a valid operation because there exists a rooted partner $N_r$ of $N_s$ such that $u$ is a parent of $v$ in $N_r$. However, in this example, the roles of $u$ and $v$ cannot be interchanged (i.e., we cannot suppress $v$ while keeping $u$) because there exists no rooted partner of $N_s$ such that $v$ is a parent of $u$ or each of $(\rho,t)$, $(t,u)$, and $(t,v)$ are cut edges, where $t$ is the child of $\rho$.

\begin{figure}[t]
    \centering
  \includegraphics[scale=0.15]{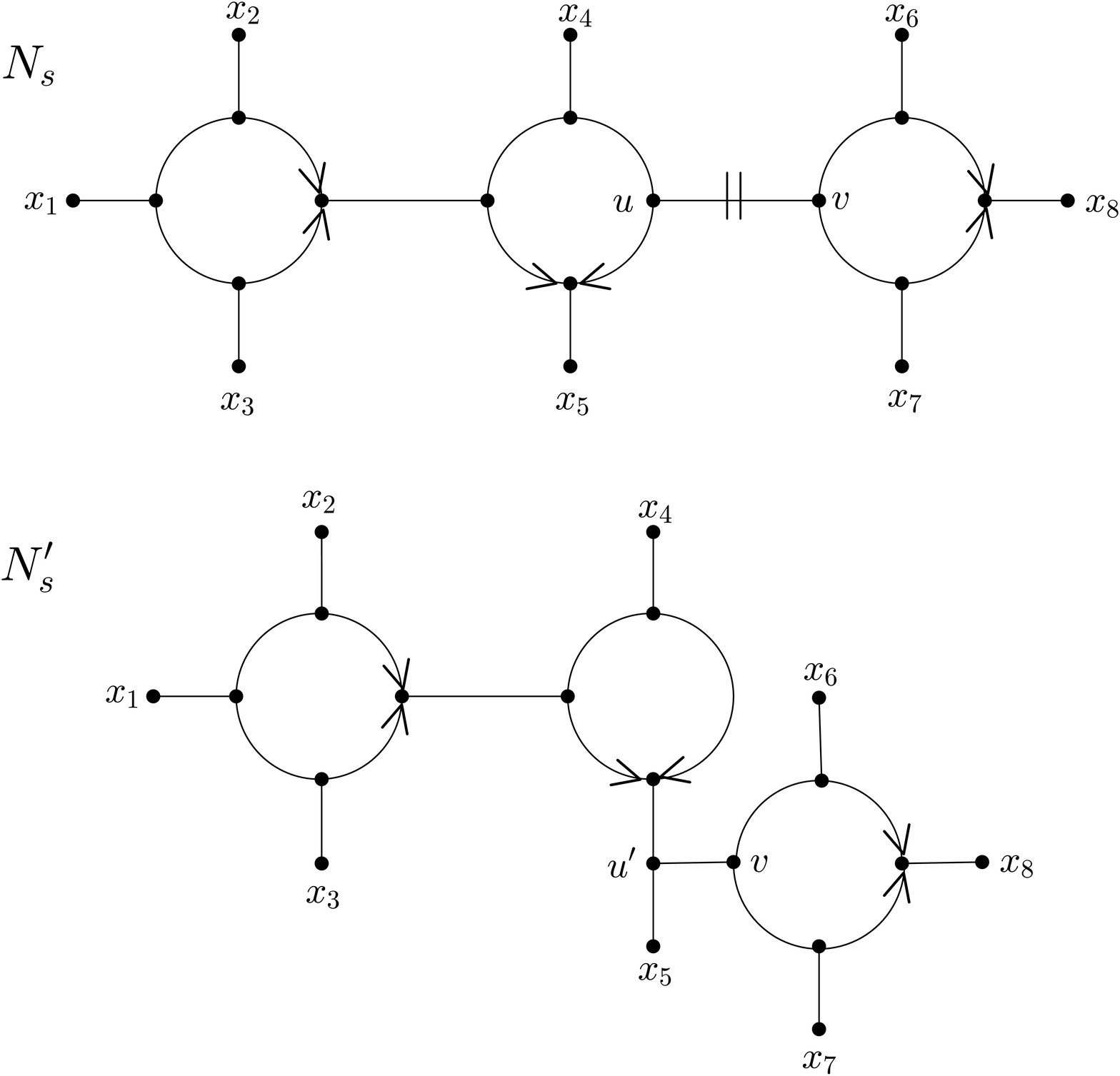}
    \caption{A semi-directed phylogenetic network $N_s$ with cut edge $e=\{u,v\}$. It can easily be checked that there exists a rooted partner of $N_s$ with $u$ being a parent of $v$. Deleting $e$, suppressing $u$, subdividing an edge of the connected component that does not contain $v$ with a new vertex $u'$, and adding a new edge $\{u',v\}$ is thus a valid CET and the semi-directed phylogenetic network $N_s'$ is obtained from $N_s$ by one such operation.} 
    \label{fig:CETsemidirected}
\end{figure}

We end this section, with several definitions that will be used throughout the remaining sections and that apply to rooted as well as to semi-directed phylogenetic networks. We call a sequence $N_0,N_1,N_2,\ldots,N_m$ of rooted phylogenetic networks on $X$ or of semi-directed phylogenetic networks on $X$ a {\em CET sequence of length $m$} if each $N_i$ with $i\in\{1,2,\ldots,m\}$ can be obtained from $N_{i-1}$ by a single CET. Let $C$ be a space of  phylogenetic networks on $X$. We say that $C$ is {\it connected} under CET if, for any pair $N$ and $N'$ of networks in $C$, there exists a CET sequence that transforms $N$ into $N'$ and every network in the sequence is in $C$. Suppose that $C$ is connected under CET. Then the {\it CET distance} between two phylogenetic networks $N$ and $N'$ in $C$ is the minimum length of a CET sequence that connects $N$ and $N'$, where every network in the sequence is in $C$. Furthermore, the {\em diameter} of $C$ under CET is the maximum CET distance over all pairs of phylogenetic networks in $C$.

\section{Connectedness of rooted level-$1$ networks}\label{sec:rooted}
In this section, we establish connectedness results under CET for spaces of rooted level-$1$ networks that have a fixed number of reticulations. These results are then used in the next section to establish analogous connectedness results for spaces of semi-directed level-$1$ networks. As we will see, almost all work goes into proving connectedness for rooted level-$1$ networks. Once the results of this section are in place, connectedness for spaces of semi-directed level-$1$ networks follows relatively easily by considering semi-directed level-$1$ networks and their rooted partners.

We start with a lemma on the number of tree vertices in a rooted phylogenetic network followed by a lemma that investigates level-$1$ networks whose cycles all have length three. To this end, recall that the root of a rooted phylogenetic network has in-degree zero and out-degree one. By translating Lemma 2.1 of~\cite{mcdiarmid2015counting} into the language of the present paper, we have the following result. 

\begin{lemma}\label{l:counts1}
Let $N_r$ be a rooted phylogenetic network on $X$. Let $k$ be the number of reticulations in $N_r$, and let $t$ be the number of tree vertices of $N_r$. Then $t=k+\vert X \vert-1$.
\end{lemma}

\begin{lemma}\label{l:counts2}
Let $N_r$ be a rooted level-$1$ network on $X$ with root $\rho$ such that each cycle has length three. Suppose that $N_r$ has exactly $k$ reticulations. Then each reticulation and tree vertex of $N_r$ is a vertex of a cycle if and only if $k=\vert X \vert-1$.
\end{lemma}

\begin{proof}
Let $t$ be the number of tree vertices of $N_r$. By Lemma~\ref{l:reticulations}, $N_r$ has at most $\vert X \vert-1$ reticulations. Furthermore, by Lemma~\ref{l:counts1}, the number of reticulations and tree vertices of $N_r$  is 
\begin{equation}\label{eq1}
k+t=k+k+\vert X \vert-1.
\end{equation}

First, assume that $k=\vert X \vert-1$. Then, Equation~(\ref{eq1}) simplifies to $k+t=3k$. Since each cycle of $N_r$ has length three, it follows that each reticulation and each tree vertex of $N_r$ is a vertex of a cycle. 

Second, assume that $k<\vert X \vert-1$. Using again Equation~(\ref{eq1}), we have $k+t>3k$. Hence, there exists a vertex $v$ in $N_r$ that is not a vertex of a cycle. By Lemma~\ref{l:ret-edges}, $v$ is a tree vertex.
\end{proof}

We now introduce what we call the {\em standard form} of a rooted level-$1$ network with precisely $k$ reticulations. This network will play a crucial role in what follows since each rooted level-$1$ network with precisely $k$ reticulations can be transformed into it by using a sequence of CETs. 
Let $N_r$ be a rooted level-$1$ network on $X$ with precisely $k$ reticulations and $\vert X \vert=n$. We say that $N_r$ is in {\it standard form} if, either $k=0$ and $N_r$ is a caterpillar, or, if $k \geq 1$ and  $N_r$ has the following properties:
\begin{enumerate}[(i)]
    \item $N_r$ contains precisely $k$ $3$-cycles. For each such cycle $C_i$ with $i\in \{1,2, \ldots, k$\}, we denote its source by $u_i$, its sink by $v_i$, and its third vertex by $p_i$.
    \item Each vertex $p_i$ is the parent of leaf $x_i$ for each $i\in\{1,2, \ldots, k\}$. 
    \item Vertex $u_1$ is the child of the root of $N_r$, and $N_r$ contains the edges $(v_i,u_{i+1})$ for each $i\in\{1,2. \ldots, k-1\}$.
    \item Leaves $x_{k+1},x_{k+2} \ldots, x_n$ are the leaves of a caterpillar $T$, such that:
    \begin{enumerate}[(a)]
        \item If $n=k+1$, leaf $x_n$ is the only leaf of $T$ and $N_r$ contains the edge $(v_k,x_n)$;
        \item If $n > k+1$, leaves $x_{k+1}, \ldots, x_n$ of $T$ are ordered such that $x_{k+1}$ and $x_{k+2}$ have the same parent and, for all $i \in \{k+2, k+3 \ldots, n-1\}$, we have that $(p_{i+1},p_i)$ is an edge in $N_r$, where $p_{i+1}$ and $p_i$ are the parents of $x_{i+1}$ and $x_i$, respectively, and such that $N_r$ contains the edge $(v_k,p_{n})$.
    \end{enumerate}
    Note that since a rooted level-$1$ network on $X$ has at most $\vert X \vert-1$ reticulations, i.e., $k \leq n-1$, we always have $n \geq k+1$, and thus one of (a) and (b) must occur.
\end{enumerate}
A generic example of a rooted level-$1$ network in standard form is depicted in Figure \ref{fig:standardform}. For fixed $X$ and fixed $k$, there is a unique rooted level-$1$ network of standard form. Continuing on from the definition of a network of standard form, we say that a rooted level-$1$ networks is of {\it standard shape} if it only differs from a network in standard form by a permutation of its leaf labels. 

\begin{figure}[t]
    \centering
    \includegraphics[scale=0.25]{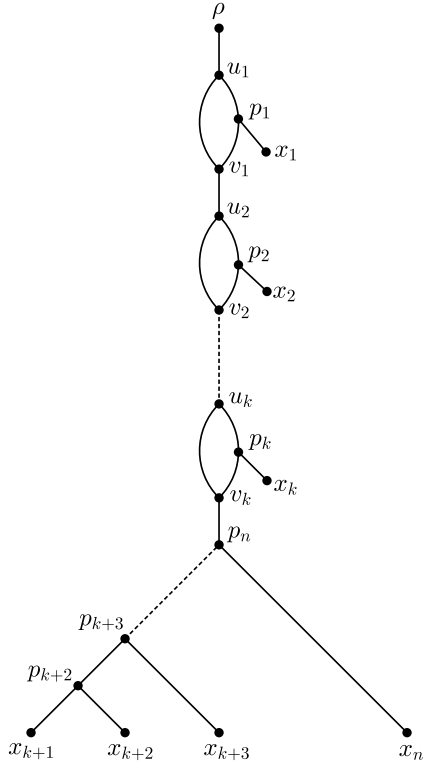}
    \caption{The rooted level-1 network on $X=\{x_1,x_2, \ldots, x_n\}$ with precisely $k$ reticulations in standard form.}
    \label{fig:standardform}
\end{figure}

To state the next lemma, we generalize the notion of level-$1$ networks. A rooted phylogenetic network on $X$ is called {\it almost level-$1$} if it has at most one 2-cycle, all other cycles have length at least three, and no two cycles have a common vertex. Moreover, we say that the space of rooted level-$1$ networks with exactly $k$ reticulations is {\it weakly connected} under CET, if, for all rooted  level-$1$ networks with exactly $k$ reticulations, $N_r$ and $N_r'$ say, there is a CET sequence connecting $N_r$ and $N_r'$ whereby every  network in the sequence is a rooted almost level-$1$ network. Now, let $N_r$ be an almost level-$1$ network on $X$. For $k\geq 1$, we say that a collection of $k$ cycles forms a \emph{chain of length $k$} of $N_r$ if there is an ordering $(C_1,C_2, \ldots, C_k)$ of these cycles such that the path from $\rho$ to $u_1$ contains only tree vertices, where $u_1$ is the source of $C_1$, and, for each $i\in\{1,2,\ldots, k-1\}$, $v_i$ is an ancestor of each vertex in $\{v_{i+1}, v_{i+2},\ldots, v_k\}$, where $v_i$ denotes the sink of $C_i$.

\begin{lemma}\label{l:standardshape}
Let $N_r$ be a rooted level-$1$ network on $X$ with precisely $k$ reticulations. Then, there exists a CET sequence of length at most $2\vert X \vert+2k$  that transforms $N_r$ into a rooted level-$1$ network $N_r^*$ on $X$ with $k$ reticulations of standard shape, whereby 
\begin{enumerate}[(i)]
    \item if $k \leq \vert X \vert-2$, every  network in the sequence is a rooted level-$1$ network on $X$ with precisely $k$ reticulations;
    \item if $k = \vert X \vert-1$, every  network in the sequence is a rooted almost level-$1$ network on $X$ with precisely $k$ reticulations.
\end{enumerate}
\end{lemma}

\noindent The high-level idea of the proof is the following: Given a rooted level-$1$ network $N_r$ with $k \geq 1$ cycles that is not of standard shape, we first transform all cycles into 3-cycles. We then arrange these 3-cycles into a chain of length $k$ and finish the transformation by moving individual leaves.

\begin{proof}[Proof of Lemma \ref{l:standardshape}]
   If $N_r$ is already in standard shape, there is nothing to show. Else, let $C_1,C_2, \ldots, C_k$ denote the cycles of $N_r$ with $k\geq 0$, 
   and let $u_i$ denote the source and $v_i$ the sink of $C_i$ for each $i \in \{1,2, \ldots, k\}$. In what follows, we generate a CET sequence of rooted almost level-$1$ networks on $X$ whereby each network in the sequence has precisely $k$ cycles. Although the length of a cycle $C_i$ may change throughout the sequence, its sink remains $v_i$. For each network in the sequence, we therefore refer to the cycle with sink $v_i$ as cycle $C_i$. 
   
   Let $(C_1, C_2, \ldots, C_k)$ be an ordering on the cycles in $N_r$ such that $C_i$ precedes $C_j$ if $u_i$ is a descendant of $u_j$ for $i < j$. 
   For each $i\in\{1,2,\ldots,k\}$ in order, we now apply a sequence of CETs to transform $C_i$ into a 3-cycle if
   $C_i$ has length at least four. Intuitively, each such CET reduces the length of $C_i$ by one. Suppose that $N_r'$ has been obtained from $N_r$ by a sequence of CETs and that cycles $C_1, C_2 \ldots, C_{i-1}$ are 3-cycles in $N_r'$. Consider the cycle $C_i$, and let $m_i$ denote its length. Further, assume that the vertices of $C_i$ are $\{u_i, v_i, s_1, s_2, \ldots, s_{m_i-2}\}$. Let $N_r^0=N_r'$ and set $j=1$. We apply the following CET to each $j \in\{1,2, \ldots, m_i-3\}$: Let $e=(s_j,t_j)$ be the cut edge incident with $s_j$. Then we obtain $N_r^j$ from $N_r^{j-1}$ by deleting $e$, suppressing $s_j$, subdividing the edge incident with $\rho$ with a new vertex $u_j'$, adding the edge $(u_j',t_j)$, and incrementing $j$ by one. By the choice of the vertices $s_j$, all moves are valid CETs and since we apply $m_i-3$ of them, no pair of parallel edges is created in the process. Moreover, when $j=m_i-2$, the size of $C_i$ is three and the process stops.  Let $N_r''$ denote the rooted level-$1$ network obtained from $N_r$ by transforming all cycles of $N_r$ into 3-cycles. It follows that each CET in the CET sequence that transforms $N_r$ into $N_r''$ cuts an edge $e=(s_j,t_j)$  in $N_r^{j-1}$ such that $t_j$ is either a leaf or a tree vertex. If $t_j$ is a tree vertex, then it has at least one descendant that is a leaf. Hence, by the chosen ordering $(C_1, C_2, \ldots, C_k)$, $N_r''$ is obtained from $N_r$ by at most $\vert X \vert$ CETs.

 Now let $(C_1',C_2',\ldots, C_k')$ be a sequence of the cycles in $N_r''$ such that $C_i'$ precedes $C_j'$ if the source $u_i'$ of $C_i'$ is  an ancestor of the source  $u_j'$ of $C_j'$ for $i< j$.
   We apply a sequence of CETs to transform $N_r''$ into a chain of 3-cycles of length $k$. If $C'_1, C_2'\ldots, C_k'$ already form a chain of 3-cycles, we apply no CET. Else assume that for some maximum $k'$ with $1 \leq k' < k$, $N_r''$ has a chain $H_{k'}$ of 3-cycles of length $k'$. Consider the minimum $j\in\{1,2,\ldots,k\}$ such that $C'_j$ is not part of $H_{k'}$. Note that $j=1$ is possible. Let $e=(t'_j,u_j')$ denote the edge directed into the source $u_j'$ of $C_j'$. By the chosen ordering, $t_j'$ is neither the root nor a reticulation of $N_r''$. We now distinguish two cases:
   \begin{enumerate}[(a)]
   \item If $k < \vert X \vert-1$, by Lemma \ref{l:counts2}, there exists at least one tree vertex in $N_r''$, $t$ say, that is not in a cycle. Let $e'=(t,c)$ denote one of its two out-going edges. We apply a sequence of three CETs. The first CET deletes $e'$, suppresses $t$, subdivides the edge $e=(t_j',u_j')$ with a new vertex $t'$, and adds the edge $(t',c)$. 
   The second CET, deletes the edge $(t',u_j')$ directed into $u_j'$, suppresses $t'$, subdivides the edge incident with $\rho$ with a new vertex $t_j''$, and adds the edge $(t_j'',u_j')$. Clearly, no parallel edges are created in this step. Finally, let $w_j'$ denote the child of $t_j''$ that is not $u_j'$. The third CET deletes the edge $(t_j'',w_j')$, suppresses $t_j''$, subdivides the cut edge incident with the sink $v_j'$ of $C_j'$ with a new vertex $t_j'''$, and adds the edge $(t_j''',w_j')$. Again, no parallel edges are created in this step. Moreover, $t_j'''$ is a tree vertex in the resulting rooted level-$1$ network that is not in a cycle. An example of this sequence is depicted in Figure \ref{fig:Lstandardshape}.
   \item If $k=\vert X \vert-1$, the procedure is similar to Case (i) except that we only perform the second and third CET since, by Lemma \ref{l:counts2}, there is no tree vertex in $N_r''$ that is {\em not} in a cycle. To be precise, the second CET move deletes the edge $(t_j',u_j')$ instead of the edge $(t,u_j')$, which implies that this CET creates a pair of parallel edges because $t_j'$ is a vertex of a cycle of length three in $N_r''$. Furthermore, applying the third CET as in Case (a) results in a rooted almost level-$1$ network with exactly one pair of parallel edges and in which  $t_j'''$ is a tree vertex that is not in a cycle. 
    \end{enumerate}
    
    Let $K$ be the subsequence of $(C'_{j+1},C_{j+2}',\ldots,C_k')$ that precisely contains each element that is not a cycle of $H_{k'}$. Since each of Cases (a) and (b) above results in a rooted almost level-$1$ network with a tree vertex that is not in a cycle, we now apply the sequence of three CETs as described in Case (a) to each cycle in $K$ in order.  It is straightforward to check that, for $k < \vert X \vert-1$, no parallel edges are created throughout the process, whereas for $k = \vert X \vert-1$ one pair of parallel edges is created by deleting $(t_j',u_j')$, but no more pairs of parallel edges arise when applying the CETs described in Case (a) to the cycles in $K$. Moreover, the first CET as described in Case (a) ensures that we can subsequently delete the edge directed into the source of a cycle in $K$ since this edge is not incident with a reticulation. 
    Let $N_r'''$ denote the rooted almost level-$1$ network obtained from $N_r''$ by the process of moving all 3-cycles as described above. Since each of Case (a) and (b) requires at most three CETs, it follows that $N_r'''$ is obtained from $N_r''$ by a sequence of at most $3k$ CETs.  Moreover, by construction, $N_r'''$ is such that the cycles $C_1, C_2 \ldots, C_k$ form a chain of cycles of length $k$ such that each cycle has length three except for  one cycle of length two if $k=\vert X \vert-1$. If $k>0$, we may assume without loss of generality that the sink $v_k$ of $C_k$ has no descendant that is a sink. Otherwise, we set $v_k$ to be the root of $N_r'''$.
    
We now complete the transformation of $N_r'''$ into a rooted level-$1$ network on $X$ of standard shape with precisely $k$ reticulations. Let $S$ be the rooted binary subtree of $N_r'''$ whose root is $v_k$, and let $X_S$ be the leaf set of $S$. If $S$ is not a caterpillar in $N_r'''$, then we apply a sequence of at most $\vert X_S \vert$ CETs that each delete a cut edge that is incident with an element in $X_S$ and that collectively transform $N_r'''$ into a rooted almost level-$1$ network on $X$ such that $v_k$ is the root of a caterpillar with leaf set $X_S$. We next distinguish again two cases. 
 First, if $k < \vert X \vert-1$, we move each leaf $x$ in $X\setminus X_S$ that is not adjacent to any 3-cycle in $N_r'''$ by deleting the edge that is directed into $x$ and subdividing the edge that is directed out of $v_k$. This transformation requires a single CET for each $x$.
 Second, if $k = \vert X \vert-1$, then $N_r'''$ contains precisely one 2-cycle. Furthermore, there is at most one leaf $x$ in $X\setminus X_S$ that is not adjacent to a 3-cycle.  If no such $x$ exists, then $\vert X_S \vert =2$ in which case we set $x$ to be one of these two leaves. Let $e=(u,v)$ be an edge of the 2-cycle in $N_r'''$. We move $x$ by deleting the edge directed into $x$ and subdividing $e$. This step requires a single CET and results in a network whose cycles all have length three.
 Let $N_r^*$ be the network obtained from $N_r'''$ as described. Then $N_r^*$ is obtained from $N_r'''$ by at most $\vert X \vert-k$ CETs. Furthermore, by construction, $N_r^*$ is a rooted level-$1$ network with precisely $k$ reticulations of standard shape. It  now follows that $N_r^*$ can be obtained from $N_r$ by a sequence of at most $\vert X \vert+3k+\vert X \vert-k=2\vert X \vert+2k$ CETs and each intermediate network is a rooted level-$1$ network with precisely $k$ reticulations if $k < \vert X \vert-1$, or a rooted almost level-$1$ network with precisely $k$ reticulations if $k = \vert X \vert-1$. This completes the proof. 
\end{proof} 

\begin{figure}[htbp]
    \centering
    \includegraphics[scale=0.225]{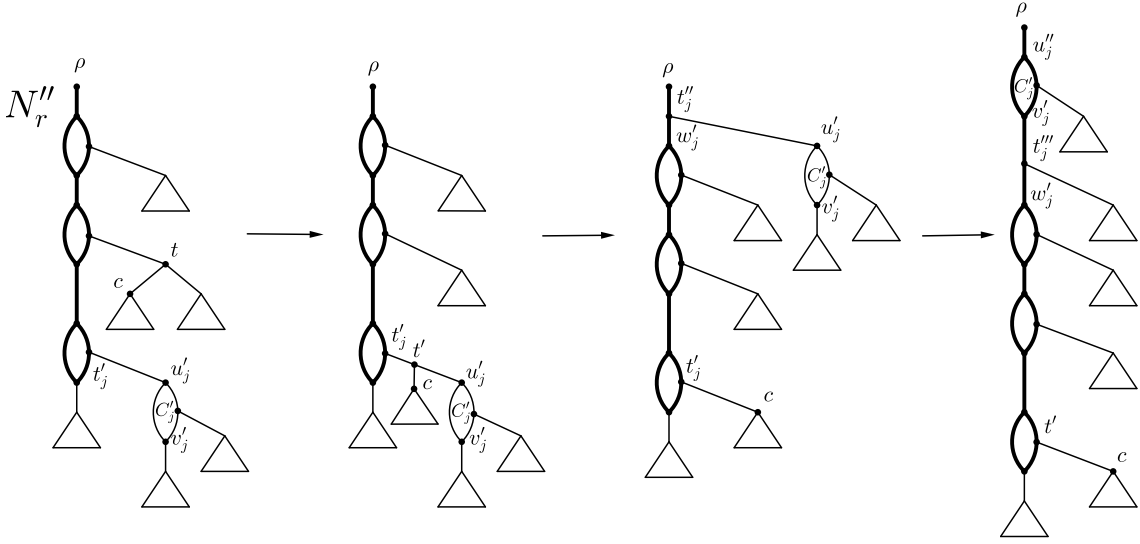}
    \caption{Sequence of three CETs as described in the proof of Lemma \ref{l:standardshape}. Triangles can be single leaves, tree-like structures, cycles, or combinations of all. Moreover, the edges connecting cycles in $N_r''$ may be paths with further branching structure, which are omitted for simplicity. The chain of 3-cycles (whose length is increased by one as a result of the sequence of CETs) is depicted in bold.}
    \label{fig:Lstandardshape}
\end{figure}

To establish the proof of the next lemma, we require some further terminology. Let $N_r$ be a rooted almost level-$1$ network on $X=\{x_1,x_2,\ldots,x_n\}$ with precisely $k$ reticulations. For each $i\in\{1,2,\ldots,k\}$, let $v_i$ be the sink of cycle $C_i$. We say that $x_i$ with $i\in\{1,2,\ldots,n\}$ is in its {\em correct position} if one of the following two conditions is satisfied.
\begin{enumerate}
    \item If $i\leq k$, then $x_i$ is adjacent to a non-sink and non-source vertex of $C_i$.
    \item If $i>k$, then $x_i$ is a leaf of a caterpillar $\sigma=(y_1,y_2,\ldots,y_{n'})$ with $n'\leq n$ that is rooted at $v_k$ such that the sequence obtained from $\sigma$ by deleting each element in $\{x_1,x_2,\ldots,x_k, x_{i+1}, x_{i+2},\ldots,x_n\}$ is equal to $(x_{k+1},x_{k+2},\ldots,x_i)$ or $(x_{k+2},x_{k+1},\ldots,x_i)$. 
\end{enumerate}

\begin{lemma}\label{l:standardform}
Let $N_r$ be a rooted level-$1$ network on $X$ with precisely $k$ reticulations such that $N_r$ is of standard shape. Then, there exists a CET sequence of length at most $3\vert X \vert$ that transforms $N_r$ into the (unique) rooted level-$1$ network on $X$ with precisely $k$ reticulations in standard form, whereby
\begin{enumerate}[(i)]
    \item if $k \leq \vert X \vert-2$, every network in the sequence is a rooted level-$1$ network on $X$ with precisely $k$ reticulations;
    \item if $k = \vert X \vert-1$, every  network in the sequence is a rooted almost level-$1$ network on $X$ with precisely $k$ reticulations. 
\end{enumerate}
\end{lemma}

\begin{proof}
Let $X=\{x_1,x_2,\ldots,x_n\}$. Furthermore, for some $k\geq 0$, let $C_1,C_2,\ldots, C_k$ denote the 3-cycles of $N_r$ where each cycle $C_i$ with $i\in\{1,2,\ldots,k\}$ has sink $v_i$. Since $N_r$ is in standard shape and only differs from a network in standard form by a permutation on the leaves, $v_i$ is an ancestor of each element in $\{v_{i+1},v_{i+2},\ldots,v_k\}$ for each $i\in\{1,2,\ldots,k-1\}$. Similar to the proof of Lemma~\ref{l:standardshape}, we generate a CET sequence of rooted almost level-$1$ networks on $X$ whereby each network in the sequence has precisely $k$ cycles. Although the length of a cycle $C_i$ may change throughout the sequence, its sink remains $v_i$. For each network in the sequence, we therefore refer to the cycle with sink $v_i$ as cycle $C_i$ and to the caterpillar with root $v_k$ as $T$.

Intuitively, we turn $N_r$ into the network of standard form by a sequence of CETs that sequentially swap the positions of leaves until every leaf is in its correct position (see Figure \ref{Fig_Lstandardform_example} for an example). To this end, each CET deletes a cut edge that is incident with a leaf $x_i$ and moves it to its correct position in the standard form, whereby we subdivide either an edge of $T$ or an edge of a cycle. 
The key idea is that if $k \leq \vert X \vert-2$, we can guarantee that no parallel edges are created, whereas if $\vert X \vert=k-1$, the creation of one pair of parallel edges is unavoidable. 

More formally, let $N_r'$ be a rooted level-$1$ network on $X$ with precisely $k$ reticulations of standard shape. Suppose that $N_r'$ has been obtained from $N_r$ by a sequence of CETs such that the leaves $x_1,x_2, \ldots, x_{i-1}$ are already in their correct position in $N_r'$ for some $i<\vert X \vert$, whereas $x_i$ is not in its correct position. If there is no such $x_i$, then all leaves are in their correct positions and $N_r'$ is already in standard form, in which case there is nothing to show. 
We now distinguish the following cases to move $x_i$ to its correct position via a sequence of CETs: 

\begin{enumerate}[(a)]
    \item If $x_i$ is a leaf of $T$ and $i>k$, we apply one CET to move $x_i$ to its correct position such that $(x_{k+1},x_{k+2},\ldots,x_{i})$ is a caterpillar. Note that the resulting network is a rooted level-$1$ network on $X$ with precisely $k$ reticulations of standard shape. 
   \item If $x_i$ is a leaf of $T$ and $i\leq k$, we distinguish two cases:
        \begin{enumerate}[(i)]
            \item If $k \leq \vert X \vert-2$, then $T$ consists of at least two leaves. In this case, we move $x_i$ to its correct position using a single CET, i.e., we move $x_i$ to the cycle $C_i$ whose sink is $v_i$. Note that this CET turns $C_i$ into a cycle of length four since $N_r'$ is of standard shape and all cycles of $N_r'$ have length exactly three.
            In particular, there exists a leaf $x_j$ with $j > i$ that is adjacent to a non-sink and non-source vertex of $C_i$. We now apply a second CET to move $x_j$ to the edge of $T$ that $x_i$ had been incident with. Intuitively, this sequence of two CETs swaps the positions of leaves $x_i$ and $x_j$ and the resulting network is again a rooted level-$1$ network with precisely $k$ reticulations of standard shape. 
            \item If $k = \vert X \vert-1$, then  $x_i$ is the only leaf in $T$ and its parent is $v_k$. Thus we cannot directly perform a CET that deletes $(v_k,x_i)$.  In this case, we consider the cycle $C_i$ whose sink is $v_i$. As $C_i$ has length exactly three, there exists a leaf $x_j$ with $j > i$ adjacent to the non-sink non-source vertex of $C_i$. Note that $x_j$ must exist since $x_i \neq x_n$, as otherwise $x_i=x_n$ would already be in its correct position. We now first move leaf $x_j$ to the edge $(v_k,x_i)$ of $T$. Then, we move $x_i$ to $C_i$. Intuitively, we again swap the positions of $x_i$ and $x_j$ using two CETs. However, while the network resulting from the second CET is a rooted level-$1$ network with precisely $k$ reticulations of standard shape, the  network resulting from the first CET contains one pair of parallel edges and is therefore a rooted almost level-$1$ network.          
        \end{enumerate}
    \item If $x_i$ is adjacent to a non-source and non-sink of a cycle $C$ of $N_r'$.
        \begin{enumerate}[(i)]
             \item If $k = \vert X \vert-1$ and $i \leq k$, we directly move $x_i$ to its correct position, i.e., we move $x_i$ to cycle $C_i$. Since $N_r'$ is a rooted level-$1$ network with precisely $k$ reticulations of standard shape and all of its cycles are 3-cycles, this move creates a pair of parallel edges and therefore a rooted almost level-$1$ network. However, for analogous reasons as above, there  exists a leaf $x_j$ with $j > i$ adjacent to a non-source and non-sink vertex of $C_i$, and we move $x_j$ to $C$. This sequence of two CETs swaps the roles of $x_i$ and $x_j$ and results in a rooted level-$1$ network with precisely $k$ reticulations of standard shape.

             \item If $k = \vert X \vert-1$ and $i > k$, then $i=n$. In this case, $x_i$ is already in its correct position, i.e., it is the single leaf of $T$ adjacent to $v_k$. This is due to the assumption that leaves $x_1,x_2, \ldots, x_{i-1}=x_{n-1}$ are already in their correct positions and $N_r'$ is a rooted level-$1$ network with precisely $k$ reticulations of standard shape. In this case, we perform no further CETs. 

             \item If $k \leq \vert X \vert-2$, the subtree $T$ of $N_r'$ contains at least two leaves. Let $x_j$ with $j > i$ be one of these leaves (which must exist for similar reasons as in the cases described above). Furthermore, if $i\leq k$ let $x_{j'}$ be the leaf that is adjacent to the non-source and non-sink vertex of $C_i$. Since $x_i$ is not in its correct position, we have $C_i\ne C$ and $x_{j'}\ne x_i$ We now first move $x_j$ to $C$, thereby turning $C$ into a 4-cycle. Next, we move $x_i$ to its correct position, i.e., we move it either to cycle $C_i$ if $i\leq k$, thereby turning $C_i$ into a 4-cycle and $C$ into a 3-cycle or to $T$ if $i>k$. If $i\leq k$, we perform one more CET and move $x_{j'}$  to the edge of $T$ that $x_j$ had been incident with. Again, this sequence of at most three CETs swaps the positions of leaves $x_i$ and $x_j$, and possibly $x_{j'}$, such that each network in the sequence is a rooted level-$1$ network with precisely $k$ reticulations and the final network is additionally of standard shape. 
    
        \end{enumerate}
   \end{enumerate}
In summary, if $k \leq \vert X \vert-2$, we  transform $N_r$ into a rooted level-1 network of standard form by a sequence of CETs, whereby every intermediate network is a rooted level-$1$ network with precisely $k$ reticulations. If $k = \vert X \vert-1$, a single pair of parallel edges might be created during the transformation and, so, every intermediate network is a rooted  almost level-$1$ network.
Moreover, since each of the cases requires at most three CETs,
it follows that the (unique) rooted level-1 network on $X$ with precisely $k$ reticulations in standard form can be obtained from $N_r$ by a sequence of at most $3\vert X \vert$ CETs. This completes the proof. 
\end{proof}

\begin{figure}[t]
    \centering
    \includegraphics[scale=0.25]{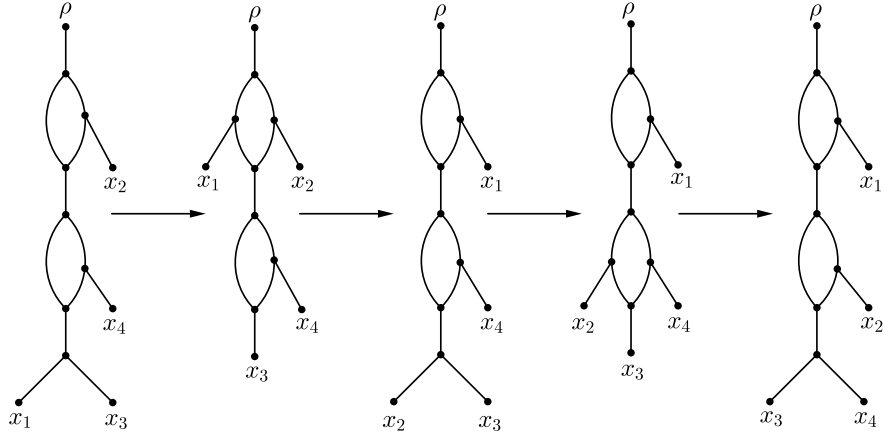}
    \caption{Sequence of CETs transforming a rooted level-$1$ network of standard shape but not standard form into a rooted level-$1$ network of standard form. 
    The first two CETs swap leaves $x_1$ and $x_2$, thereby moving $x_1$ to its correct position. The next two CETs then move $x_2$ to its correct position by swapping leaves $x_2$ and $x_4$. The resulting network is already of standard form, implying that no more CETs are required.}
    \label{Fig_Lstandardform_example}
\end{figure}

We now establish the main result of this section.

\begin{theorem} \label{cor:rootedconnectedness}
Let $k$ be a fixed non-negative integer. If $k \leq \vert X \vert-2$, then the space of rooted level-$1$ networks on $X$ with exactly $k$ reticulations is connected under CET. Otherwise, if $k=\vert X \vert-1$, then the space of rooted level-$1$ networks on $X$ with exactly $k$ reticulations is weakly connected under CET. Moreover, in both cases, the diameter of the space of rooted level-$1$ networks on $X$ with exactly $k$ reticulations is at most $O(\vert X \vert+k)$ under CET.
\end{theorem}

\begin{proof}
Let $N_r$ and $N_r'$ be two rooted level-$1$ networks on $X$ with exactly $k$ reticulations. First, if $k\leq 
\vert X \vert-2$ then, by Lemmas~\ref{l:standardshape} and~\ref{l:standardform}, $N_r$ (resp. $N_r'$) can be transformed into the rooted level-$1$ network on $X$ with precisely $k$ reticulations in standard form such that each intermediate network is level-1 and has exactly $k$ reticulations. Hence, if $k\leq 
\vert X \vert-2$, it follows from the reversibility of CET that the space of  rooted level-$1$ networks with exactly $k$ reticulations is connected.
Second, if $k=
\vert X \vert-1$ then, again by Lemmas~\ref{l:standardshape} and~\ref{l:standardform}, $N_r$ (resp. $N_r'$) can be transformed into the rooted level-$1$ network on $X$ with precisely $k$ reticulations in standard form such that each intermediate network is almost level-1 and has exactly $k$ reticulations. Hence, if $k=
\vert X \vert-1$, then the space of  rooted level-$1$ networks with exactly $k$ reticulations is weakly connected. Moreover, applying Lemmas~\ref{l:standardshape} and~\ref{l:standardform} one more time, it requires at most $2\vert X \vert+2k+3\vert X \vert=5\vert X \vert+2k$ CETs to transform each of $N_r$ and $N_r'$ into the unique rooted level-$1$ network on $X$ with exactly $k$ reticulations in standard form. Hence, if $k=\vert X \vert-1$ (resp. $k<\vert X \vert-1$), then there exists a CET sequence of length at most $10\vert X \vert+4k$ that connects $N_r$ and $N_r'$ in the space of all rooted level-$1$ networks on $X$ with exactly $k$ reticulations (resp. in the space of all rooted almost level-$1$ networks on $X$ with exactly $k$ reticulations). In both cases, the diameter is therefore $O(\vert X \vert+k)$.
\end{proof}

\noindent We remark in passing that Theorem~\ref{cor:rootedconnectedness} strengthens a previous result on the connectedness of the space of rooted level-$1$ networks on $X$ with exactly $k$ reticulations.  In particular, Klawitter showed in~\cite{klawitter2020spaces} that this space is connected if one allows for $k$ pairs of parallel edges, whereas our result requires at most one pair of parallel edges.

\section{Connectedness of semi-directed level-$1$ networks} \label{sec:connectedness}
\subsection{Connectedness for networks with a fixed number of reticulations} \label{sec:withintier}
In this section, we use the results established in Section~\ref{sec:rooted} to establish connectedness results under CET for spaces of semi-directed level-$1$ networks with a fixed number of reticulations. Similar to the previous section, we start by generalizing the notion of semi-directed level-$1$ networks. A semi-directed phylogenetic network is called {\it almost level-$1$} if it has a rooted partner that is almost level-$1$. Thus, a semi-directed almost level-$1$ network has at most two cycles of length two and no loop, or at most one loop and no cycle of length two. Furthermore, we say that the space of semi-directed level-$1$ networks with exactly $k$ reticulations is {\it weakly connected} under CET, if, for all semi-directed  level-$1$ networks with exactly $k$ reticulations, $N_s$ and $N_s'$ say, there is a CET sequence connecting $N_s$ and $N_s'$ whereby every  network in the sequence is a semi-directed almost level-$1$ network.

The main result of this section is the following theorem. 
\begin{theorem}\label{t:semidirected_connected}
Let $k$ be a fixed non-negative integer. If $k \leq \vert X \vert-2$, then, the space of semi-directed level-$1$ networks on $X$ with exactly $k$ reticulations is connected under CET. Otherwise, if $k=\vert X \vert-1$, then the space of semi-directed level-$1$ networks on $X$ with exactly $k$ reticulations is weakly connected under CET. 
Moreover, in both cases, the diameter of the space of semi-directed level-$1$ networks on $X$ with exactly $k$ reticulations is at most $O(\vert X \vert+k)$ under CET.
\end{theorem}

To motivate the allowance of parallel edges in establishing connectedness results for semi-directed level-$1$ networks, note that if $k = \vert X \vert-1$, the space of semi-directed level-$1$ networks on $X$ with precisely $k$ reticulations is not necessarily connected. As an example, consider the space of semi-directed level-$1$ networks with $\vert X \vert=2$ and $k=1$. Let $N_s$ be the semi-directed level-$1$ network  depicted in Figure \ref{Fig_NrNsNu}, and let $N_s'$ be the semi-directed level-$1$ network obtained from $N_s'$ by interchanging $x_1$ and $x_2$. Then, $N_s \ncong N_s'$ and there exists no CET sequence that transforms $N_s$ into $N_s'$, whereby every network in the sequence is a semi-directed level-$1$ network with one reticulation. However, it is possible to transform $N_s$ into $N_s'$ by a sequence of two CETs, whereby the network obtained from $N_s$ by the first CET is a semi-directed {\em almost} level-$1$ network with one reticulation.

Before proving Theorem~\ref{t:semidirected_connected},  we establish a connection between a sequence of CETs connecting two semi-directed almost level-$1$ networks and such a sequence connecting their rooted partners.

\begin{lemma} \label{l:CETrooted_semidirected}
Let $N_s^1$ and $N_s^2$ be two distinct semi-directed almost level-$1$ networks, and let $N^1_r$ and $N^2_r$ be two almost level-$1$ rooted partners of $N_s^1$ and $N_s^2$, respectively. If  $N^2_r$ can be obtained from $N^1_r$ by a single CET, then $N_s^2$ can be obtained from $N_s^1$ by one CET. 
\end{lemma}

\begin{proof}
Suppose that $N^2_r$ can be obtained from $N^1_r$ by a single CET. Let $e=(u,v)$ be the cut edge of $N^1_r$ that is deleted in obtaining $N^2_r$ from $N^1_r$. Let $M$ and $M'$ be the two connected subnetworks that result from deleting $e$ and suppressing $u$, where $M$ contains $\rho$ and $M'$ contains $v$. 
Furthermore, let $f$ be the edge of $M$ that is subdivided with a new vertex $u'$ in obtaining $N^2_r$ from $M$ and $M'$ be adding the edge $(u',v)$. Observe that $f$ is also an edge of $N_r^1$. Moreover, by definition of a CET,  $u$ is not a reticulation and $u\ne \rho$.  Now, let $t$ be the unique child of $\rho$ in $N_r^1$. Since $N_r^1\ncong N_r^2$, it follows that $e$ and $f$ cannot both be incident with $t$. To complete the proof, we consider three cases. 

First, assume that neither $e$ nor $f$ is incident with $t$. By Lemma~\ref{l:cut-edge}, $\{u,v\}$ is a cut edge of $N_s^1$. Moreover, since $N_s^1$ is obtained from $N_r^1$ by deleting $\rho$, suppressing $t$, and undirecting all tree edges, $f$ is also an edge of $N_s^1$. It now follows that $N_s^2$ can be obtained from $N_s^1$ by the CET that deletes $\{u,v\}$, suppresses $u$, subdivides $f$ with a new vertex $u'$, and joins the two vertices $u'$ and $v$ with a new edge. 

Second, assume that $e$ is incident with $t$. Then $t=u$ and all three edges that are incident with $t$ are cut edges of $N_r^1$. Let $w$ be the second child of $t$ that is not $v$. If follows from Lemma~\ref{l:cut-edge}, that $\{v,w\}$ is a cut edge of $N_s^1$. Furthermore, as $f$ is not incident with $t$, an argument analogous to that used in the first case implies that $f$ is also an edge of  $N_s^1$. Hence, $N_s^2$ can be obtained from $N_s^1$ by the CET that deletes $\{v,w\}$, suppresses $w$, subdivides $f$ with a new vertex $u'$, and joins the two vertices $u'$ and $v$ with a new edge.

Third, assume that $f$ is incident with $t$.
As before, $\{u,v\}$ is a cut edge of $N_s^1$ by Lemma~\ref{l:cut-edge}. If $t$ is not the source of a cycle, let $w$ and $w'$ be the two children of $t$ in $N_r^1$, Then $\{w,w'\}$ is an edge in $N_s^1$. Hence, $N_s^2$ can be obtained from $N_s^1$ by the CET that deletes $\{u,v\}$, suppresses $u$, subdivides $\{w,w'\}$ with a new vertex $u'$, and joins the two vertices $u'$ and $v$ with a new edge. On the other hand, if $t$ is the source of a cycle in $N_r^1$, let $(t,w)$ and $(t,w')$ be the two edges that are directed out of $t$.  It follows that $\{w,w'\}$, $(w,w')$, or $(w',w)$ is an edge of $N_s^1$ depending on whether or not one of $w$ and $w'$ is a reticulation in $N_s^1$.  Since $N_r^1$ is almost level-$1$, we may have $w=w'$, in which case $(w,w')$ is a loop. Thus, $N_s^2$ can be obtained from $N_s^1$ by the CET that deletes $\{u,v\}$, suppresses $u$, subdivides $\{w,w'\}$ with a new vertex $u'$, and joins the two vertices $u'$ and $v$ with a new edge. Additionally, if $\{w,w'\}$ is a loop in $N_s^1$, then one of the two resulting parallel edges that each join $u'$ and $w$ is initially undirected and therefore directed into $w$ in $N_s^2$.

It now follows that, for all three cases, $N_s^2$ can be obtained from $N_s^1$ by one CET; thereby establishing the lemma.
\end{proof}

We are now in a position to prove Theorem~\ref{t:semidirected_connected}.
\begin{proof}[Proof of Theorem \ref{t:semidirected_connected}]
Let $N_s$ and $N_s'$ be two semi-directed level-$1$ networks on $X$ that each have exactly $k$ reticulations. Furthermore, let $N_r$ and $N_r'$ be a level-$1$ rooted partner of $N_s$ and $N_s'$, respectively. By Theorem~\ref{cor:rootedconnectedness} and its proof, there exists a CET sequence $$N_r \cong N_r^1,N_r^2,\cdots,N_r^{m-1},N_r^m \cong N_r'$$ with $m\leq 10\vert X \vert+4k$  that connects $N_r$ and $N_r'$ such that each network in the sequence is either a rooted almost level-$1$ network on $X$ and with exactly $k$ reticulations if $k=\vert X \vert-1$ or a rooted level-$1$ network on $X$ with exactly $k$ reticulations if $k< \vert X \vert-1$. For each $i\in\{2,3,\ldots,m-1\}$, let $N_s^i$ be the semi-directed network on $X$ that is obtained from $N_r^i$ by deleting $\rho$, suppressing the resulting vertex of in-degree zero and out-degree two, and undirecting all tree edges. By construction, $N_s^i$ has exactly $k$ reticulations and $N_r^i$ is a rooted partner of $N_s^i$.

Set $N_s^1=N_s$ and $N_s^m=N_s'$. Then, for each $i\in\{1,2,\ldots,m\}$, $N_s^i$ is level-$1$ (resp. almost level-$1$) if and only if $N_r^i$ is level-$1$ (resp. almost level-$1$). Now consider $N_s^i$ and $N_s^{i+1}$ for each $i\in\{1,2,\ldots,m-1\}$. We may have $N_s^i\cong N_s^{i+1}$. It follows from Lemma~\ref{l:CETrooted_semidirected} that $N_s^{i+1}$ can be obtained from $N_s^i$ by at most one CET. Hence, there exists a sequence of at most $m$ CETs that connects $N_s$ and $N_s'$ such that each network in the sequence is either a semi-directed almost level-$1$ network with exactly $k$ reticulations if $k= \vert X \vert-1$ or a semi-directed level-$1$ network on $X$ with exactly $k$ reticulations if $k\leq \vert X \vert-2$. The theorem now follows. 
\end{proof}

\subsection{Connectedness for networks with a varying number of reticulations} \label{sec:acrosstiers}
In this section, we show that the space of semi-directed level-$1$ networks on a fixed leaf set is connected under CET. 
We start with an observation that we freely use throughout this section. For a semi-directed phylogenetic network $N_s$ with no reticulation, the definition of a CET on $N_s$ coincides with that of a subtree prune and regraft  (SPR) operation for unrooted phylogenetic trees. To be precise,  an {\it unrooted binary phylogenetic $X$-tree} $T$ is an undirected tree whose leaves are bijectively labeled with $X$ and whose internal vertices all have degree three. Under the subtree prune and regraft operation, it is well-known that the space of all unrooted phylogenetic trees on a fixed leaf set is connected~\cite{allen2001subtree,maddison1991discovery}.  

In addition to the CET as defined in Section~\ref{sec:CET}, we next introduce two operations that change the number of reticulations in semi-directed phylogenetic network $N_s$ by one.\\

\noindent {\bf CET$^-$.} Let $e=(u,v)$ be a reticulation edge of $N_s$ such that, if $u\ne v$, then $u$ is not a reticulation. If $(u,v)$ is a loop, obtain a network $N_s'$ from $N_s$
by deleting $u$. Otherwise, obtain $N_s'$ from $N_s$ 
by undirecting the edge that is directed into $v$ and not $e$, deleting $e$, and suppressing $u$ and $v$.\\

\noindent If $e$ is a loop, then $N_s$ has a unique rooted partner and it follows that  the neighbor of $u$ in $N_s$ is not a reticulation. Hence, regardless of whether $e$ is a loop in $N_s$ or not, $N_s'$ is a semi-directed phylogenetic network on $X$.\\

\noindent {\bf CET$^+$.} Let $e$ be an edge of $N_s$. Obtain a network $N$ from $N_s$ in one of the following two ways: (i) Subdivide $e$ with a new vertex $v$, add the edge $\{u,v\}$, where $u$ is a new vertex, and add the (directed) loop $(u,u)$; or (ii) subdivide $e$ with a new vertex $v$, subdivide an edge in the resulting network  with a new vertex $u$, add the new edge $(u,v)$, and direct one of the two other edges incident with $v$ into $v$.\\

\noindent In contrast to CET$^-$, observe that a CET$^+$ does not necessarily result in a semi-directed phylogenetic network. For example, if $N_s$ contains a loop and $N$ is obtained from $N_s$ by a CET$^+$ as described in (i), then $N$ contains two loops and is not a semi-directed phylogenetic network. 

Now let $N_s$ and $N_s'$ be two semi-directed phylogenetic networks on $X$. If $N_s'$ can be obtained from $N_s$ by a single CET$^+$ (resp. CET$^-$), then $N_s$ can by obtained from $N_s'$ by a single CET$^-$ (resp. CET$^+$). Furthermore, we say that $N_s'$ can be obtained from $N_s$ by a single {\it extended CET} if it can be obtained by applying exactly one of CET, CET$^-$, and CET$^+$ to $N_s$.  Similar to the CET distance, we refer to the minimum number of extended CETs  that are required to transform $N_s$ into $N_s'$ as the {\it extended CET distance} between $N_s$ and $N_s'$,

\begin{theorem}\label{t:across-tiers}
The space of all semi-directed level-$1$ networks on $X$ is connected under extended CET.
\end{theorem}

\begin{proof}
Let $N_s$ and $N_s'$ be two semi-directed level-$1$ networks on $X$, and let $k=r(N_s)$ and  $k'=r(N_s')$. Furthermore, let $(v_1,v_2,\ldots,v_k)$ be an ordering on the reticulations of $N_s$ and, similarly, let $(v_1',v_2',\ldots,v_k')$ be an ordering on the reticulations of $N_s'$. 

Now, setting $N_s^0=N_s$, repeat the following operation $k$ times for each $i\in\{1,2,\ldots,k\}$ in order. Obtain a network $N_s^i$ from $N_s^{i-1}$ by applying a CET$^-$ to a reticulation edge $(u_i,v_i)$ that is incident with $v_i$. Since $N_s^0$ is level-$1$, it follows that $u_1$ is not a reticulation. Thus $N_s^1$ is a semi-directed level-$1$ network on $X$. Repeating this argument, it follows that
each $N_s^i$ with $i\in\{0,1,2,\ldots,k\}$ is a semi-directed level-$1$ network on $X$ and $N_s^k$ is an unrooted phylogenetic tree on $X$.  Let $T_s=N_s^k$, and let $T_s'$ be an unrooted phylogenetic tree on $X$ obtained from $N_s'$ by applying $k'$  CET$^-$ in an analogous way. Since $T_s'$ can be obtained from $T_s$ by a sequence of subtree prune and regraft operations, it follows that  that $T_s'$ can be obtained from $T_s$ by a sequence of CETs and each tree in the sequence is an unrooted phylogenetic tree on $X$. The theorem now follows from the reversibility of CET, CET$^+$, and CET$^-$.
\end{proof}

The next theorem  is similar to Theorem~\ref{t:across-tiers}  and establishes connectedness for the larger space of semi-directed phylogenetic networks on a fixed leaf set.

\begin{theorem}\label{t:across2}
The space of all semi-directed phylogenetic networks on $X$ is connected under extended CET.
\end{theorem}

\begin{proof}
Let $N_s$ and $N_s'$ be two semi-directed phylogenetic networks on $X$
with $k=r(N_s)$ and  $k'=r(N_s')$. Let $N_r$ and $N_r'$ be a  rooted partner of $N_s$ and $N_s'$, respectively.  Furthermore, let $(v_1,v_2,\ldots,v_k)$ be an ordering on the reticulations of $N_s$ such that, for all $i,j\in \{1,2,\ldots,k\}$ with $i<j$, $v_i$ is not a descendant of $v_j$ in $N_r$. Similarly, let $(v_1',v_2',\ldots,v_{k'}')$ be an ordering on the reticulations of $N_s'$ such that, for all distinct $i,j\in \{1,2,\ldots,k'\}$ with $i<j$, $v_i'$ is not a descendant of $v_j'$ in $N_r'$. 
The  theorem can now be established analogously to Theorem~\ref{t:across-tiers}.
The more constrained ordering of the reticulations of $N_s$ and $N_s'$ in comparison to that used in the proof of Theorem~\ref{t:across-tiers} guarantees that each CET$^-$  is applied to a reticulation edge $(u,v)$ of a semi-directed phylogenetic network on $X$ such that $u$ is not a reticulation.
\end{proof}

The next corollary follows immediately from Theorems~\ref{t:across-tiers} and~\ref{t:across2}, and the fact that each extended CET is reversible.

\begin{corollary}
The extended CET distance is a metric on the space of all semi-directed phylogenetic networks as well as on all semi-directed level-$1$ networks on $X$.
\end{corollary}

\subsection{Connectedness using CET$_1$ moves} \label{Sec:CET1}
In the following, we consider CETs that operate \enquote{locally} in the sense that when a cut edge $\{u,v\}$ of a semi-directed phylogenetic network is deleted, the connected component containing $v$ is re-attached via the introduction of a new cut edge in close proximity to its original position (see formal definitions below). We then show that every CET that satisfies a mild constraint can be translated into a sequence of these local CETs.

Using similar terminology as \cite{Gambette2017}, we now define CET$_1$ moves. Let $N_s$ be a semi-directed phylogenetic network on $X$. First, when a CET deletes a cut edge $\{u,v\}$, we refer to the two edges incident with $u$ in $N_s$ that are different from the edge $\{u,v\}$ as the \emph{donor edges}, and to the edge that is subdivided by $u'$ in $N_s$ prior to adjoining $u'$ and $v$ with a new edge as the \emph{recipient edge}.
Then, a \emph{CET$_1$} is a CET applied to $N_s$ such that the recipient edge is incident with one of the two donor edges. As an example, the CET depicted in Figure \ref{fig:CETsemidirected} is a CET$_1$ since the recipient edge, i.e., the edge incident with leaf $x_5$, is also incident with one of the two donor edges incident with $u$. If we had instead subdivided the edge incident with leaf $x_3$ by $u'$ and then added the edge $\{u',v\}$, the resulting CET would not have been a CET$_1$.

Before stating the main proposition of this section, we require two more technical concepts. Let $N_s$ be a semi-directed almost level-$1$ network that contains at least one pair of parallel edges and at least one 3-cycle. We say that a CET applied to $N_s$ {\it changes the location of a pair of parallel edges} if it deletes a cut edge $e$ whose two donor edges are edges of a 3-cycle (turning this 3-cycle into a 2-cycle) and whose recipient edge is an edge of a 2-cycle (turning this 2-cycle into a 3-cycle). Similarly, if $N_s$ contains (i) precisely one loop and at least one 3-cycle, or (ii) precisely two pairs of parallel edges, we say that a CET applied to $N_s$ {\it exchanges a loop for two pairs of parallel edges or vice versa} if it (i) deletes a cut edge $e$ whose two donor edges are edges of a 3-cycle (turning this 3-cycle into a 2-cycle) and whose recipient edge is the loop of $N_s$ (turning the loop into a second 2-cycle), or (ii) deletes a cut edge $e$ whose two donor edges form a 2-cycle (turning this 2-cycle into a loop) and whose recipient edge is an edge of a 2-cycle (turning this 2-cycle into a 3-cycle).

We now show that if $N_s$ and $N_s'$ are two semi-directed (almost) level-$1$ networks on $X$ with exactly $k$ reticulations that are one CET apart such that the CET does not change the location of a pair of parallel edges, and does not exchange a loop for two pairs of parallel edges or vice versa, then there is also a sequence of CET$_1$ moves connecting $N_s$ and $N_s'$, whereby every network in the sequence is a semi-directed (almost) level-$1$ network with exactly $k$ reticulations. Note that the restriction of not changing the location of a pair of parallel edges or exchanging a loop for a pair of parallel edges is required to ensure that every network in the sequence is indeed an (almost) level-$1$ network.

\begin{proposition}\label{p:CET1moves}
Let $N_s$ and $N_s'$ be two semi-directed level-$1$ networks on $X$ and with exactly $k$ reticulations if $k < \vert X \vert-1$, respectively two semi-directed almost level-$1$ networks with exactly $k$ reticulations if $k=\vert X \vert-1$, such that $N_s'$ can be obtained from $N_s$ by a single CET that neither changes the location of a pair of parallel edges nor exchanges a loop for two pairs of parallel edges or vice versa. Then, there exists a CET$_1$ sequence transforming $N_s$ into $N_s'$ such that each network in the sequence is level-$1$ and has exactly $k$ reticulations if $k<\vert X \vert-1$ or each network in the sequence is almost level-$1$ and has exactly $k$ reticulations if $k=\vert X \vert-1$.
\end{proposition}

\begin{proof}
We first show that there is a CET$_1$ sequence connecting $N_s$ and $N_s'$, whereby every network in the sequence is a semi-directed network with precisely $k$ reticulations. Let $e=\{u,v\}$ be the edge of $N_s$ that is deleted in obtaining $N_s'$ from $N_s$ by a single CET. Furthermore, let $e'=\{p,q\}$ (respectively, $e'=(p,q)$ if $e'$ is directed) denote the recipient edge in $N_s$. If $e'$ is incident with one of the two donor edges incident with $u$ in $N_s$, the CET to obtain $N_s'$ from $N_s$ is a CET$_1$ and there is nothing to show. Thus, assume that $e'$ is not incident with one of the two donor edges. As $N_s$ is connected, there exists an undirected path $P$ between $u$ and $p$. Let $\{u,u_1\}, \{u_1, u_2\}, \ldots, \{u_{l-1}, u_l\}, \{u_l,u_{l+1}\}$ be the sequence of edges of $P$ with $\{u_l,u_{l+1}\} = \{p,q\}$. To ease reading, we view all edges of $P$ as being undirected regardless of whether they are tree or reticulation edges of $N_s$. We now argue that the CET  transforming $N_s$ into $N_s'$ can also be realized as a CET$_1$ sequence along $P$. More precisely, the first CET$_1$  consists of deleting $e=\{u,v\}$ and suppressing $u$, subdividing the edge $\{u_1,u_2\}$ with a new vertex $u^1$, and introducing the edge $\{u^1,v\}$. Because every CET$_1$ is also a CET, this results in a semi-directed phylogenetic network $N_s^1$ with cut edge $\{u^1,v\}$ and precisely $k$ reticulations. Moreover, by construction, $N_s^1$ has a rooted partner, $N_r^1$ say, such that $u^1$ is the parent of $v$ in $N_r^1$ or there exist three cut edges $(\rho,t), (t,u^1)$, and $(t,v)$ in $N_r^1$. Lastly, observe that $\{u_2,u_3\}, \{u_3, u_4\}, \ldots, \{u_l,u_{l+1}\}$ is  a path in $N_s^1$. We now perform a second CET$_1$, whereby we delete $\{u^1,v\}$ and suppress $u^1$ in $N_s^1$, subdivide the edge $\{u_2,u_3\}$ with a new vertex $u^2$, and introduce the edge $\{u^2,v\}$. By construction, this results in a semi-directed network $N_s^2$ with cut edge $\{u^2,v\}$ and precisely $k$ reticulations, where $u^2$ and $v$ are again such that $u^2$ is a parent of $v$ in the rooted partner $N_r^2$ of $N_s^2$, or $N_r^2$ contains the three cut edges $(\rho,t), (t,u^2)$, and $(t,v)$. Furthermore, $\{u_3,u_4\}, \{u_4,u_5\}, \ldots,  \{u_l,u_{l+1}\}$ is  a path in $N_s^2$. If $l>2$, we next apply a CET$_1$ to $\{u^2,v\}$ in $N_s^2$ with recipient edge $\{u_3,u_4\}$ and repeat. As $P$ consists of a finite number of edges, this process will eventually lead to a semi-directed network $N_s^l$ obtained from the semi-directed network $N_s^{l-1}$ by deleting the edge $\{u^{l-1},v\}$, suppressing $u^{l-1}$, subdividing the edge $\{u_l,u_{l+1}\} = \{p,q\}$ with a new vertex $u^l$, and adding the edge $\{u^l,v\}$. Since all vertices $u^i$ with $1 \leq i < l$ introduced during this process are immediately suppressed in subsequent steps, clearly $N_s^l \cong N_s'$, which completes the first part of the proof. 

It remains to argue that every network in the sequence is level-$1$ if $k < \vert X \vert-1$ and is almost level-$1$ if $k=\vert X \vert-1$. Consider the above  CET$_1$ sequence $N_s, N_s^1, N_s^2, \ldots, N^l_s \cong N_s'$ transforming $N_s$ into $N_s'$. 
We first consider the CET$_1$ transforming $N_s$ into $N_s^1$ and distinguish two cases:
\begin{enumerate}[(i)]
    \item If $k < \vert X \vert-1$, $N_s$ and $N_s'$ are semi-directed level-$1$ networks and contain at most one pair of parallel edges each and no loop. First, suppose that $N_s$ contains one pair of parallel edges. Since the CET transforming $N_s$ into $N_s'$ by assumption does not change the location of a pair of parallel edges, this implies that the donor edges of $N_s$ cannot be part of a 3-cycle.
    Hence, when deleting $e=\{u,v\}$ from $N_s$ to obtain $N_s^1$, no additional pair of parallel edges is created.  Second,  suppose that $N_s$ contains no pair of parallel edges. Then $N_s^1$ contains at most one pair of parallel edges. Thus, in both cases, $N_s^1$ is also level-$1$.
    \item If $k = \vert X \vert-1$, $N_s$ and $N_s'$ are semi-directed almost level-$1$ networks and each contain at most one loop and no pair of parallel edges, or at most two pairs of parallel edges but no loop. Assume for the sake of a contradiction that deleting $e=\{u,v\}$ from $N_s$ to obtain $N_s^1$ results in $N_s^1$ not being almost level-$1$, i.e., containing either three pairs of parallel edges, two loops, or one loop and a pair of parallel edges, while deleting $e=\{u,v\}$ from $N_s$ to obtain $N_s'$ results in $N_s'$ containing at most one loop and no pair of parallel edges, or at most two pairs of parallel edges but no loop.
    \begin{itemize}
        \item If $N_s^1$ contains three pairs of parallel edges, $N_s$ contains two pairs of parallel edges and the donor edges of $N_s$ are edges of a 3-cycle. Since $e=\{u,v\}$ is also deleted when transforming $N_s$ into $N_s'$, either $N_s'$ also contains three pairs of parallel edges, a contradiction to the fact that $N_s'$ is almost level-$1$, or the recipient edge for the CET from $N_s$ into $N_s'$ is an edge of a 2-cycle, which is also a contradiction, since the CET by assumption does not change the location of a pair of parallel edges. 
        \item If $N_s^1$ contains two loops, $N_s$ contains at least one loop and at least one pair of parallel edges. This contradicts the fact that $N_s$ is an almost level-$1$ network. 
        \item Finally, if $N_s^1$ contains one loop and one pair of parallel edges, then either (a) $N_s$ contains one loop and the donor edges of $N_s$ are edges of a 3-cycle, or (b) $N_s$ contains two pairs of parallel edges, and the donor edges of $N_s$ are edges of such a pair. Again, as $e$ is also deleted to obtain $N_s^1$ from $N_s$, either $N_s'$ also contains one loop and one pair of parallel edges, contradicting the fact that $N_s'$ is almost level-$1$, or the CET from $N_s$ to $N_s'$ is such that (a) the loop of $N_s$ is exchanged for two pairs of parallel edges, or (b) the two pairs of parallel edges of $N_s$ are exchanged for a loop. Both cases contradict the fact that the CET  transforming $N_s$ into $N_s'$ does not exchange a loop for two pairs of parallel edges or vice versa. 
    \end{itemize}
As all three cases lead to a contradiction, $N_s^1$ is an almost level-$1$ network. 
\end{enumerate}
Now, consider $i \in \{1,2, \ldots, l-1\}$ and the CET$_1$ transforming $N_s^i$ into $N_s^{i+1}$. Suppose that deleting the cut edge $\{u^i,v\}$ introduces an excessive loop or pair of parallel edges such that $N_s^{i+1}$ is not level-$1$ if $k < \vert X \vert-1$ or is not almost level-$1$ if $k=\vert X \vert-1$. Since $\{u^i,v\}$ was newly introduced when transforming $N_s^{i-1}$ into $N_s^{i}$, this new loop or pair of parallel edges must have already existed in $N_s^{i-1}$ and thus ultimately in $N_s$. Thus, it cannot be excessive and $N_s^{i+1}$ is a level-$1$, respectively almost level-$1$ network. This completes the proof.
\end{proof}

A CET$_1$ may be interpreted as an NNI move for semi-directed phylogenetic networks. In particular, a CET$_1$ move on such a network with no reticulation coincides with an NNI move on an unrooted phylogenetic tree.
Revisiting the CET sequences used in the proofs of Lemma \ref{l:standardshape} and \ref{l:standardform} to establish (weak) connectedness under CET for rooted level-$1$ networks with precisely $k$ reticulations and translating these sequences into their semi-directed counterparts to establish (weak) connectedness under CET for semi-directed level-$1$ networks with precisely $k$ reticulations, we notice that no CET changes the location of a pair of parallel edges or exchanges a loop for two pairs of parallel edges (or vice versa). Thus, the conditions of Proposition \ref{p:CET1moves} are satisfied and the next corollary follows from Theorem \ref{t:semidirected_connected}, where the definition of {\it weakly connected under CET$_1$} is analogous to that of weakly connected under CET.

\begin{corollary}\label{c:CET1}
Under CET$_1$, the space of semi-directed level-$1$ networks on $X$ with exactly $k$ reticulations is connected if $k \leq \vert X \vert-2$ and is weakly connected if $k = \vert X \vert-1$. 
\end{corollary}

As mentioned in the introduction, the authors of~\cite{SolisLemus2016} conjectured that the five types of moves employed in SNaQ are sufficient to guarantee connectedness of the space of semi-directed level-$1$ networks with a fixed leaf set. While one of these five types increases the number of reticulations by one, no move decreases this number. Hence, SNaQ must effectively guarantee connectedness of the space of semi-directed level-$1$ network with a fixed number of reticulations and leaf set, because once a search through the space of semi-directed level-$1$ networks reaches a network with $k$ reticulations every network that is investigated later in the search has at least $k$ reticulations. Although a precise definition of SNaQ's fourth move type, called NNI move on a tree edge, is unfortunately missing in~\cite{SolisLemus2016}, Corollary~\ref{c:CET1} suggests that the space of level-$1$ networks with a fixed number of reticulations and leaf set is connected under the five moves employed in SNaQ if the authors additionally allow for NNI moves on a reticulation edge. Our results also suggests that, if $k=\vert X \vert-1$, then semi-directed level-$1$ networks that allow for at most two 2-cycles and a single loop need to be considered when searching for an optimal network although, as noted in~\cite{SolisLemus2016}, reticulations in a 2-cycle and certain other types of short cycles with small adjacent subnetworks are either not detectable or their parameters are not all identifiable.

\section{Concluding remarks} \label{sec:conclusion}
In this paper, we have introduced a new rearrangement operation on semi-directed phylogenetic networks, called CET, that  transforms any semi-directed level-$1$ network with precisely $k$ reticulations into any other such network with the same set of leaves. Moreover, we have introduced two additional operations, CET$^+$ and CET$^-$, that allow to move between  semi-directed phylogenetic networks and between semi-directed level-$1$ networks with a fixed leaf set and an arbitrary number of reticulations. While CET moves have a similar flavor as SPR and rSPR moves on unrooted, respectively rooted phylogenetic trees and networks~\cite{allen2001subtree,bordewich2005computational,bordewich2017lost,Gambette2017}, we have also shown that any CET can be translated into a sequence of more local CET$_1$ moves, which are similar to NNI moves studied on phylogenetic trees and networks~\cite{Gambette2017,Huber2015,Janssen2019,robinson1971comparison}. Such CET$_1$ moves essentially coincide with moves that are used in the popular network inference software PhyloNetworks \cite{SolisLemus2016, SolisLemus2017} up to a slight relaxation of one of their moves. Thus, our theoretical results on the connectedness of the space of semi-directed level-$1$ networks provide some level of assurance that a (locally) optimal semi-directed level-$1$ network can be reached from any such starting network in a hill-climbing search.

While our main focus has been to establish connectedness and diameter results for the space of semi-directed level-$1$ networks with a fixed number of reticulations and leaf set, there are several open questions to explore in future research. For instance, it would be interesting to analyze the computational complexity of determining the CET distance between any two semi-directed phylogenetic networks. It would also be interesting to analyze further properties of the space of semi-directed phylogenetic networks on a fixed leaf set or subspaces of it such as the radius of the space. 
Finally, one could ask which of the results presented in this paper carry over to unrooted phylogenetic networks.

\bigskip

\noindent{\bf Acknowledgements.} We thank  Sungsik Kong, Laura Kubatko, and Claudia Sol\'is-Lemus for helpful discussions. SL thanks the New Zealand Marsden Fund for their financial support.

\bibliographystyle{abbrvnat}
\bibliography{References.bib}

\begin{thebibliography}{32}
\providecommand{\natexlab}[1]{#1}
\providecommand{\url}[1]{\texttt{#1}}
\expandafter\ifx\csname urlstyle\endcsname\relax
  \providecommand{\doi}[1]{doi: #1}\else
  \providecommand{\doi}{doi: \begingroup \urlstyle{rm}\Url}\fi

\bibitem[Allen and Steel(2001)]{allen2001subtree}
B.~L. Allen and M.~Steel.
\newblock Subtree transfer operations and their induced metrics on evolutionary
  trees.
\newblock \emph{Annals of Combinatorics}, 5\penalty0 (1):\penalty0 1--15, 2001.

\bibitem[Allman et~al.(2019)Allman, Ba{\~n}os, and Rhodes]{allman2019nanuq}
E.~S. Allman, H.~Ba{\~n}os, and J.~A. Rhodes.
\newblock \mbox{NANUQ}: a method for inferring species networks from gene trees
  under the coalescent model.
\newblock \emph{Algorithms for Molecular Biology}, 14\penalty0 (1):\penalty0
  1--25, 2019.

\bibitem[Allman et~al.(2022)Allman, Ba{\~{n}}os, and Rhodes]{Allman2022}
E.~S. Allman, H.~Ba{\~{n}}os, and J.~A. Rhodes.
\newblock Identifiability of species network topologies from genomic sequences
  using the {logDet} distance.
\newblock \emph{Journal of Mathematical Biology}, 84:\penalty0 35, 2022.

\bibitem[Ardiyansyah(2021)]{Ardiyansyah2021}
M.~Ardiyansyah.
\newblock Distinguishing level-2 phylogenetic networks using phylogenetic
  invariants.
\newblock \emph{arXiv:2104.12479}, 2021.

\bibitem[Ba{\~{n}}os(2018)]{Banos2018}
H.~Ba{\~{n}}os.
\newblock Identifying species network features from gene tree quartets under
  the coalescent model.
\newblock \emph{Bulletin of Mathematical Biology}, 81\penalty0 (2):\penalty0
  494--534, 2018.

\bibitem[Bordewich and Semple(2005)]{bordewich2005computational}
M.~Bordewich and C.~Semple.
\newblock On the computational complexity of the rooted subtree prune and
  regraft distance.
\newblock \emph{Annals of Combinatorics}, 8\penalty0 (4):\penalty0 409--423,
  2005.

\bibitem[Bordewich et~al.(2017)Bordewich, Linz, and Semple]{bordewich2017lost}
M.~Bordewich, S.~Linz, and C.~Semple.
\newblock Lost in space? \mbox{G}eneralising subtree prune and regraft to
  spaces of phylogenetic networks.
\newblock \emph{Journal of Theoretical Biology}, 423:\penalty0 1--12, 2017.

\bibitem[Cardona et~al.(2008)Cardona, Rossell{\'o}, and
  Valiente]{cardona2008comparison}
G.~Cardona, F.~Rossell{\'o}, and G.~Valiente.
\newblock Comparison of tree-child phylogenetic networks.
\newblock \emph{IEEE/ACM Transactions on Computational Biology and
  Bioinformatics}, 6\penalty0 (4):\penalty0 552--569, 2008.

\bibitem[Erd{\H{o}}s et~al.(2021)Erd{\H{o}}s, Francis, and Mezei]{Erdos2021}
P.~L. Erd{\H{o}}s, A.~Francis, and T.~R. Mezei.
\newblock Rooted {NNI} moves and distance-1 tail moves on tree-based
  phylogenetic networks.
\newblock \emph{Discrete Applied Mathematics}, 294:\penalty0 205--213, 2021.

\bibitem[Francis et~al.(2017)Francis, Huber, Moulton, and Wu]{Francis2017}
A.~Francis, K.~T. Huber, V.~Moulton, and T.~Wu.
\newblock Bounds for phylogenetic network space metrics.
\newblock \emph{Journal of Mathematical Biology}, 76\penalty0 (5):\penalty0
  1229--1248, 2017.

\bibitem[Gambette et~al.(2017)Gambette, van Iersel, Jones, Lafond, Pardi, and
  Scornavacca]{Gambette2017}
P.~Gambette, L.~van Iersel, M.~Jones, M.~Lafond, F.~Pardi, and C.~Scornavacca.
\newblock Rearrangement moves on rooted phylogenetic networks.
\newblock \emph{{PLoS} Computational Biology}, 13\penalty0 (8):\penalty0
  e1005611, 2017.

\bibitem[Gross and Long(2018)]{Gross2018}
E.~Gross and C.~Long.
\newblock Distinguishing {P}hylogenetic networks.
\newblock \emph{{SIAM} Journal on Applied Algebra and Geometry}, 2\penalty0
  (1):\penalty0 72--93, 2018.

\bibitem[Gross et~al.(2021)Gross, van Iersel, Janssen, Jones, Long, and
  Murakami]{Gross2020}
E.~Gross, L.~van Iersel, R.~Janssen, M.~Jones, C.~Long, and Y.~Murakami.
\newblock Distinguishing level-1 phylogenetic networks on the basis of data
  generated by {M}arkov processes.
\newblock \emph{Journal of Mathematical Biology}, 83:\penalty0 32, 2021.

\bibitem[Hein et~al.(1996)Hein, Jiang, Wang, and Zhang]{hein1996complexity}
J.~Hein, T.~Jiang, L.~Wang, and K.~Zhang.
\newblock On the complexity of comparing evolutionary trees.
\newblock \emph{Discrete Applied Mathematics}, 71\penalty0 (1-3):\penalty0
  153--169, 1996.

\bibitem[Hollering and Sullivant(2021)]{Hollering2021}
B.~Hollering and S.~Sullivant.
\newblock Identifiability in phylogenetics using algebraic matroids.
\newblock \emph{Journal of Symbolic Computation}, 104:\penalty0 142--158, 2021.

\bibitem[Huber et~al.(2015)Huber, Linz, Moulton, and Wu]{Huber2015}
K.~T. Huber, S.~Linz, V.~Moulton, and T.~Wu.
\newblock Spaces of phylogenetic networks from generalized nearest-neighbor
  interchange operations.
\newblock \emph{Journal of Mathematical Biology}, 72\penalty0 (3):\penalty0
  699--725, 2015.

\bibitem[Huber et~al.(2016)Huber, Moulton, and Wu]{Huber2016}
K.~T. Huber, V.~Moulton, and T.~Wu.
\newblock Transforming phylogenetic networks: {M}oving beyond tree space.
\newblock \emph{Journal of Theoretical Biology}, 404:\penalty0 30--39, 2016.

\bibitem[Huber et~al.(2022)Huber, van Iersel, Janssen, Jones, Moulton,
  Murakami, and Semple]{huber2022orienting}
K.~T. Huber, L.~van Iersel, R.~Janssen, M.~Jones, V.~Moulton, Y.~Murakami, and
  C.~Semple.
\newblock Orienting undirected phylogenetic networks.
\newblock \emph{arXiv:1906.07430}, 2022.

\bibitem[Huson et~al.(2010)Huson, Rupp, and Scornavacca]{huson2010phylogenetic}
D.~H. Huson, R.~Rupp, and C.~Scornavacca.
\newblock \emph{Phylogenetic networks: concepts, algorithms and applications}.
\newblock Cambridge University Press, 2010.

\bibitem[Janssen(2021)]{Janssen2021}
R.~Janssen.
\newblock Heading in the right direction? {U}sing head moves to traverse
  phylogenetic network space.
\newblock \emph{Journal of Graph Algorithms and Applications}, 25\penalty0
  (1):\penalty0 263--310, 2021.

\bibitem[Janssen and Klawitter(2019)]{Janssen2019}
R.~Janssen and J.~Klawitter.
\newblock Rearrangement operations on unrooted phylogenetic networks.
\newblock \emph{Theory and Applications of Graphs}, 22\penalty0 (1):\penalty0
  1--31, 2019.

\bibitem[Janssen et~al.(2018)Janssen, Jones, Erd{\H{o}}s, van Iersel, and
  Scornavacca]{Janssen2018}
R.~Janssen, M.~Jones, P.~L. Erd{\H{o}}s, L.~van Iersel, and C.~Scornavacca.
\newblock Exploring the tiers of rooted phylogenetic network space using tail
  moves.
\newblock \emph{Bulletin of Mathematical Biology}, 80\penalty0 (8):\penalty0
  2177--2208, 2018.

\bibitem[Klawitter(2018)]{Klawitter2018}
J.~Klawitter.
\newblock The {SNPR} neighbourhood of tree-child networks.
\newblock \emph{Journal of Graph Algorithms and Applications}, 22\penalty0
  (2):\penalty0 329--355, 2018.

\bibitem[Klawitter(2020)]{klawitter2020spaces}
J.~Klawitter.
\newblock \emph{Spaces of phylogenetic networks}.
\newblock PhD thesis, University of Auckland, 2020.

\bibitem[Kong et~al.(2022)Kong, Swofford, and Kubatko]{kong2022inference}
S.~Kong, D.~Swofford, and L.~Kubatko.
\newblock Inference of phylogenetic networks from sequence data using composite
  likelihood.
\newblock \emph{bioRxiv}, 2022.

\bibitem[Maddison(1991)]{maddison1991discovery}
D.~R. Maddison.
\newblock The discovery and importance of multiple islands of most-parsimonious
  trees.
\newblock \emph{Systematic Biology}, 40\penalty0 (3):\penalty0 315--328, 1991.

\bibitem[McDiarmid et~al.(2015)McDiarmid, Semple, and
  Welsh]{mcdiarmid2015counting}
C.~McDiarmid, C.~Semple, and D.~Welsh.
\newblock Counting phylogenetic networks.
\newblock \emph{Annals of Combinatorics}, 19\penalty0 (1):\penalty0 205--224,
  2015.

\bibitem[Robinson(1971)]{robinson1971comparison}
D.~F. Robinson.
\newblock Comparison of labeled trees with valency three.
\newblock \emph{Journal of Combinatorial Theory, Series B}, 11\penalty0
  (2):\penalty0 105--119, 1971.

\bibitem[Sol{\'{\i}}s-Lemus and An{\'{e}}(2016)]{SolisLemus2016}
C.~Sol{\'{\i}}s-Lemus and C.~An{\'{e}}.
\newblock Inferring phylogenetic networks with maximum pseudolikelihood under
  incomplete lineage sorting.
\newblock \emph{{PLoS} Genetics}, 12\penalty0 (3):\penalty0 e1005896, 2016.

\bibitem[Sol{\'{\i}}s-Lemus et~al.(2017)Sol{\'{\i}}s-Lemus, Bastide, and
  An{\'{e}}]{SolisLemus2017}
C.~Sol{\'{\i}}s-Lemus, P.~Bastide, and C.~An{\'{e}}.
\newblock {PhyloNetworks}: {A} package for phylogenetic networks.
\newblock \emph{Molecular Biology and Evolution}, 34\penalty0 (12):\penalty0
  3292--3298, 2017.

\bibitem[Sol\'{i}s-Lemus et~al.(2020)Sol\'{i}s-Lemus, Coen, and
  An\'{e}]{Solis-Lemus2020}
C.~Sol\'{i}s-Lemus, A.~Coen, and C.~An\'{e}.
\newblock On the identifiability of phylogenetic networks under a
  pseudolikelihood model.
\newblock \emph{arXiv:2010.01758}, 2020.

\bibitem[Xu and An{\'e}(2023)]{Xu2021}
J.~Xu and C.~An{\'e}.
\newblock Identifiability of local and global features of phylogenetic networks
  from average distances.
\newblock \emph{Journal of Mathematical Biology}, 86:\penalty0 12, 2023.

\end{thebibliography}

\end{document}